\documentclass[a4paper]{article}

\usepackage{amsmath,amssymb}
\usepackage{amsthm}
\usepackage{amscd}
\usepackage{bm}
\usepackage{listings}
\usepackage[all]{xy}

\theoremstyle{definition}
\newtheorem{theorem}{Theorem}[section]
\newtheorem{prop}[theorem]{Proposition}
\newtheorem{lemma}[theorem]{Lemma}

\newtheorem{definition}[theorem]{Definition}

\newtheorem{remark}[theorem]{Remark}
\newtheorem*{note}{Notation}

\numberwithin{theorem}{section}

\newcommand\Ker{\operatorname{Ker}}
\newcommand\Image{\operatorname{Im}}
\newcommand\Aut{\operatorname{Aut}}
\newcommand\GL{\operatorname{GL}}
\newcommand\SL{\operatorname{SL}}

\newcommand\PGL{\operatorname{PGL}}
\newcommand\PSL{\operatorname{PSL}}
\newcommand\PBD{\operatorname{PBD}}
\newcommand\Sing{\operatorname{Sing}}

\title{Projective plane curves whose automorphism groups are simple and primitive}

\author{Yusuke Yoshida}

\date{}

\begin{document}

\maketitle

\begin{abstract}
We study complex projective plane curves with a given group of automorphisms. Let $G$ be a simple primitive subgroup of $\PGL(3, \mathbb{C})$, which is isomorphic to $\mathfrak{A}_{6}, \mathfrak{A}_{5}$ or $\PSL(2, \mathbb{F}_{7})$. We obtain a necessary and sufficient condition on $d$ for the existence of a nonsingular projective plane curve of degree $d$ invariant under $G$. We also study an analogous problem on integral curves.
\end{abstract}

\section{Introduction}

Automorphism groups of algebraic curves have long been studied. For example, Hurwitz found an upper bound on the order of the automorphism group of a curve of a given genus (\cite{Hurwitz}). In the case of plane curves, there are more detailed studies on automorphism groups. One important fact here is that an automorphism of a smooth projective plane curve of degree greater than or equal to $4$ uniquely extends to an automorphism of the projective plane. Hence, the automorphism group of such a curve is isomorphic to a subgroup of $\PGL(3, \mathbb{C})$.

Recently, Harui obtained the following result concerning the classification of automorphism groups of smooth plane curves.

\begin{theorem}\label{thm:Harui}
[Harui, Theorem 2.3]
Let $C$ be a smooth plane curve of degree $d \geq 4$, and $G$ a subgroup of $\Aut C$. Then one of the following holds:
\begin{enumerate}
\item[(a-i)] $G$ fixes a point on $C$, and $G$ is cyclic.
\item[(a-ii)] $G$ fixes a point not lying on $C$, and up to conjugation there is a commutative diagram
\[
\xymatrix@M=8pt{
1 \ar[r] & \mathbb{C}^{\times} \ar[r] \ar@{}[d]|\bigcup & \PBD(2, 1) \ar[r]^-\pi \ar@{}[d]|\bigcup & \PGL(2, \mathbb{C}) \ar[r] \ar@{}[d]|\bigcup & 1 \\
1 \ar[r] & N \ar[r] & G \ar[r] & G' \ar[r] & 1
}
\]
with exact rows where the subgroup $\PBD(2, 1)$ of $\PGL(3, \mathbb{C})$ is defined by
\[
\PBD(2, 1) := \left\{ \left[ 
\begin{array}{cc}
A & O \\
O & \alpha \\
\end{array}
\right]
\in \PGL(3, \mathbb{C}) \middle| A \in \GL(2, \mathbb{C}), \alpha \in \mathbb{C}^{\times} \right\},
\]
$N$ is cyclic and $G'$ is isomorphic to a cyclic group $\mathbb{Z}/m\mathbb{Z}$, a dihedral group $\mathcal{D}_{2m}$, the tetrahedral group $\mathfrak{A}_{4}$, the octahedral group $\mathfrak{S}_{4}$ or the icosahedral group $\mathfrak{A}_{5}$.
\item[(b-i)] $(C, G)$ is a \textit{descendant} of the Fermat curve $F_{d}$ defined by $x^{d} + y^{d} + z^{d} = 0$. (For the definition of a descendant, see [Harui, Definition 2.2].)

In particular, $G$ is conjugate to a subgroup of $\Aut F_{d}$.
\item[(b-ii)] $(C, G)$ is a descendant of the Klein curve $K_{d}$ defined by $x^{d-1}y + y^{d-1}z + z^{d-1}x = 0$.

In particular, $G$ is conjugate to a subgroup of $\Aut K_{d}$.
\item[(c)] $G$ is a finite \textit{primitive} subgroup of $\PGL(3, \mathbb{C})$. (For the definition of a primitive subgroup of $\PGL(3, \mathbb{C})$ and more details, see Definition \ref{def of primitive}.)

In this case, $G$ is conjugate to one of the following groups:
\begin{itemize}
\item The Valentiner group $\mathcal{V}$, which is isomorphic to $\mathfrak{A}_{6}$.
\item The icosahedral group $\mathcal{I}$, which is isomorphic to $\mathfrak{A}_{5}$.
\item The Klein group $\mathcal{K}$, which is isomorphic to $\PSL(2, \mathbb{F}_{7})$.
\item The Hessian group $H_{216}$ of order $216$ or its subgroup of order $36$ or $72$.
\end{itemize}
\end{enumerate}
\end{theorem}

There are also more concrete studies of the automorphism groups. For example, the list of all automorphism groups of projective plane curves of degree $5$ was given by E. Badr and F. Bars (\cite{BB}). Beyond that, it seems that we still do not have the list of the automorphism groups of projective plane curves of a given degree.

We look at the problem from a different point of view. Fix a group $G$ in the classification, consider nonsingular (resp. integral) projective plane curves $C$ invariant under $G$. Then we want to find all possible values of the degree of such a curve $C$.

For a group $G$ of type (a-i) or (a-ii) in the classification, invariant curves under $G$ are related with the study of Galois points and are actively studied. (See e.g. \cite{HMO}.)

For a pair $(C, G)$ as in (b-i) or (b-ii) in the classification, $C$ is defined by 
\[
x^{d} + y^{d} + z^{d} + \sum\limits_{i + j + k = d, \max\{ i, j, k\} < d} c_{i j k} x^{i} y^{j} z^{k} = 0
\]
or
\[
x^{d-1}y + y^{d-1}z + z^{d-1}x + \sum\limits_{i + j + k = d, \max\{ i, j, k\} < d-1} c_{i j k} x^{i} y^{j} z^{k} = 0.
\]
and $G$ acts on $C$. Thus it can be said that we have a good understanding of these cases.

In this paper, we focus on groups of type (c) in the classification and study invariant curves under each of them except for subgroups of $H_{216}$ and its sub groups. As the first main result, we determine all degrees of nonsingular projective plane curves invariant under each of the groups $\mathcal{V} \cong \mathfrak{A}_{6}$, $\mathcal{I} \cong \mathfrak{A}_{5}$ and $\mathcal{K} \cong \PSL(2, \mathbb{F}_{7})$:
\begin{theorem}
Let $d$ be a positive integer.
\begin{enumerate}
\item[(1)] There exists a nonsingular projective plane curve of degree $d$ whose automorphism group is equal to $\mathcal{V}$ if and only if $d \equiv 0, 6$ or $12 \mod 30$.
\item[(2)] There exists a nonsingular projective plane curve of degree $d$ invariant under $\mathcal{I}$ if and only if $d \equiv 0, 2$ or $6 \mod 10$.
\item[(3)] There exists a nonsingular projective plane curve of degree $d$ whose automorphism group is equal to $\mathcal{K}$ if and only if $d \equiv 0, 4$ or $6 \mod 14$.
\end{enumerate}
\end{theorem}
We note that if $C$ is invariant under $\mathcal{V}$ or $\mathcal{K}$ then $\Aut C$ is equal to $\mathcal{V}$ or $\mathcal{K}$, respectively.

Next, we give a result on integral (i.e. irreducible and reduced) curves invariant under each of the groups above. We again find all degrees of such curves:
\begin{theorem}
Let $d$ be a positive integer.
\begin{enumerate}
\item[(1)] There exists an integral projective plane curve of degree $d$ invariant under $\mathcal{V}$ if and only if $d$ is a multiple of $6$, $d \neq 18$ and $d \neq 24$.
\item[(2)] There exists an integral projective plane curve of degree $d$ invariant under $\mathcal{I}$ if and only if $d$ is even and is neither $4, 8$ nor $14$.
\item[(3)] There exists an integral projective plane curve of degree $d$ invariant under $\mathcal{K}$ if and only if $d$ is even and is neither $2, 8, 12, 16$ nor $22$.
\end{enumerate}
\end{theorem}

We prove these results in the following way. The groups $\mathcal{V}$, $\mathcal{I}$ and $\mathcal{K}$ have been classically well studied as reflection groups. (Some results are summarized in \cite{ST}, for example.) If $G$ is a simple primitive finite subgroup of $\PGL(3, \mathbb{C})$, which is conjugate to $\mathcal{V}$, $\mathcal{I}$ or $\mathcal{K}$, then we can take a lift $\widetilde{G}$ of $G$ in $\GL(3, \mathbb{C})$ such that a curve invariant under $G$ is defined by a polynomial invariant under $\widetilde{G}$. In addition, the invariant rings were studied for $\mathfrak{A}_{6}$ in \cite{Wiman} and for $\PSL(2, \mathbb{F}_{7})$ in \cite{Klein} around the end of the 19th century. By regarding $\mathfrak{A}_{5}$ as a subgroup of $\mathfrak{A}_{6}$, we can give the invariant ring of $\mathfrak{A}_{5}$.

For a given degree $d$, we can thus describe the linear system $(\mathfrak{d}_{G})_{d}$ of $G$-invariant curves of degree $d$ using explicitly given invariant polynomials $F_{G}$, $\Phi_{G}$ and $\Psi_{G}$. We can translate conditions on the base locus of $(\mathfrak{d}_{G})_{d}$ to congruence relations on $d$. Then we look at singularities of a general element of $(\mathfrak{d}_{G})_{d}$ with the help of Bertini's theorem and study when it is nonsingular at the base points. 

We next consider integral curves. The problem is to show that there exist integral invariant curves of degree $d$ in the theorem. In this case, it can be shown that a general element $C$ is reduced in the same way. Since any nonsingular plane curve is integral, it suffices to consider degrees $d$ for which $(\mathfrak{d}_{G})_{d}$ has only singular elements. It turns out that the singular points form a $G$-orbit, hence are of the same type. We can also give the number of singularities and the type of singularities of a general member. From this information, we see, if a general element were reducible, it would have too much singularity.

The organization of this paper is as follows. In Section 2, we recall a number of facts needed for the discussion of invariant curves. We give a description of homogeneous polynomials invariant under the group $\mathcal{V}$, $\mathcal{I}$ or $\mathcal{K}$ following \cite{Crass} and \cite{Elkies}. In Section 3, we first give a condition for the existence of a nonsingular invariant curve in a form that is valid for any of $\mathcal{V}$, $\mathcal{I}$ and $\mathcal{K}$, and then translate these conditions for each group. In Section 4, we study integral invariant curves. We look at singularities of a general invariant curve of a degree where there is no nonsingular invariant curve and prove that it cannot be reducible. In Appendix, we give some codes in SINGULAR to check a few calculations in Section 2.

\section{Simple primitive finite subgroups of $\PGL(3, \mathbb{C})$ and invariant curves}

In this section, we give a description of the groups $\mathcal{V}$, $\mathcal{I}$ and $\mathcal{K}$, their invariant rings and the invariant curves.

First, we give basic definitions related to invariant homogeneous polynomials.

\begin{note}
Let $\bm{x}$ be the coordinate $(x, y, z)$ in $\mathbb{C}^{3}$. We often identify $\bm{x}$ with the row vector $(x {\ } y {\ } z)$ or the column vector $\left(\begin{array}{c} x \\ y \\ z \end{array}\right)$.
\end{note}

\begin{definition}
\begin{enumerate}
\item[(1)] For a matrix $A \in \GL(3, \mathbb{C})$ and a homogeneous polynomial $f(x, y, z) \in \mathbb{C}[x, y, z]$, we set
\[
(f^{A})(x, y, z) := f\left( \bm{x} {}^{t}\!A \right) = f\left( A \bm{x} \right).
\]
Note that this is an action from the right.

\item[(2)] Let $G$ be a subgroup of $\GL(3, \mathbb{C})$.

We say that a homogeneous polynomial $f$ is \textit{semi-invariant} under $G$ if there exists a group homomorphism $\chi : G \rightarrow \mathbb{C}^{\times}$ such that $(f^{A})(x, y, z) = \chi(A) f(x, y, z)$ for any $A \in G$. In particular, if $\chi$ is trivial, then we say that $f$ is \textit{invariant} under $G$ (or $G$-\textit{invariant}).

\item[(3)] For an equivalence class of a matrix $[A] \in \PGL(3, \mathbb{C})$ and a point $(a : b : c) \in \mathbb{P}^{2}$, we put
\[
[A] \cdot (a : b : c) := \left[A \cdot \left(\begin{array}{c} a \\ b \\ c \end{array}\right) \right].
\]
Then we identify $\PGL(3, \mathbb{C})$ with the automorphism group $\Aut \mathbb{P}^{2}$.

\item[(4)] Let $C$ be a projective plane curve, by which we mean a nonzero effective divisor of $\mathbb{P}^{2}$.

For a subgroup $G < \PGL(3, \mathbb{C})$, we say that $C$ is \textit{invariant} under $G$ (or $G$-\textit{invariant}) if $\sigma_{*} C = C$ for any $\sigma \in G$. In particular, if $C$ is reduced, then $C$ is invariant under $G$ if $G \cdot C = C$ as a subset of $\mathbb{P}^{2}$.

\end{enumerate}
\end{definition}

\begin{note}
We write the zero set
\[
V(f) := \{ P \in \mathbb{P}^{2} \mid f(P) = 0 \}
\]
for a homogeneous polynomial $f(x, y, z) \in \mathbb{C}[x, y, z]$.
\end{note}

\begin{remark}
Let $G$ be a subgroup of $\PGL(3, \mathbb{C})$ and $\pi : \SL(3, \mathbb{C}) \rightarrow \PGL(3, \mathbb{C})$ denote the natural homomorphism. If a subgroup $\widetilde{G}$ of $\SL(3, \mathbb{C})$ satisfies $\pi(\widetilde{G}) = G$, then we call $\widetilde{G}$ a \textit{lift} of $G$.

Assume that $C$ is a projective plane curve defined by a homogeneous polynomial $f \in \mathbb{C}[x, y, z]$, i.e., $C = V(f)$. Then $C$ is invariant under $G$ if and only if $f$ is semi-invariant under $\widetilde{G}$.
\end{remark}

The groups which we are going to consider are simple finite subgroups of $\PGL(3, \mathbb{C})$. The following proposition holds.

\begin{prop}
Let $G$ be a finite subgroup of $\PGL(3, \mathbb{C})$ and $\pi : \SL(3, \mathbb{C}) \rightarrow \PGL(3, \mathbb{C})$ the natural surjective homomorphism.

If $G$ is simple and nonabelian, then either there is no nontrivial group homomorphism $\pi^{-1}(G) \rightarrow \mathbb{C}^{\times}$ or $\pi$ splits.
\end{prop}

\begin{proof}
Since $G$ is simple and nonabelian, it has no nontrivial abelian character. In fact, the assertion holds under the latter condition. Let $\iota : \mathbb{Z}/3\mathbb{Z} \rightarrow \widetilde{G}$ be the map $\overline{a} \mapsto \rho^{a} I_{3}$ where $\rho = e^{\frac{2\pi i}{3}}$ and $\chi : \pi^{-1}(G) \rightarrow \mathbb{C}^{\times}$ a homomorphism.

We define a homomorphism $\varphi : \mathbb{Z}/3\mathbb{Z} \rightarrow \mathbb{C}^{\times}$ by $\varphi = \chi \circ \iota$. Then we have a commutative diagram
\[
\xymatrix
{
1 \ar[r] & \mathbb{Z}/3\mathbb{Z} \ar[r]^-\iota \ar[rd]^-\varphi & \pi^{-1}(G) \ar[r]^-\pi \ar[d]^-\chi & G \ar[r] & 1,\\
&& \mathbb{C}^{\times} &&
}
\]
where the first row is exact.

First, assume that $\varphi$ is trivial. Then $\chi$ induces a group homomorphism $\overline{\chi} : G \rightarrow \mathbb{C}^{\times}$. By assumption, $\overline{\chi}$ is trivial. Hence, $\chi$ is trivial.

On the other hand, suppose that $\varphi$ is nontrivial. This means $\Image \iota \not\subset \Ker \chi$, hence, $\Image \iota \cap \Ker \chi = 1$, and $\pi |_{\Ker \chi}$ is injective. The quotient group $G/\pi(\Ker\chi)$ is isomorphic to $\Image\chi/\Image\varphi$ which is abelian. Since $G$ has no nontrivial abelian quotient, we obtain $G = \pi(\Ker\chi)$. Therefore, $\pi$ induces an isomorphism $\Ker \chi \cong G$, and $\pi$ is split.
\end{proof}

From this proposition, we obtain the following lemma.

\begin{lemma}\label{semi-invariant prop}
Let $G$ be a simple and nonabelian finite subgroup of $\PGL(3, \mathbb{C})$. There is a lift $\widetilde{G}$ of $G$ such that any invariant curve under $G$ is defined by a homogeneous polynomial invariant under $\widetilde{G}$.
\end{lemma}

\begin{proof}
Let $\pi : \SL(3, \mathbb{C}) \rightarrow \PGL(3, \mathbb{C})$ be the natural surjective homomorphism. If $\pi^{-1}(G)$ has no nontrivial abelian character, then we take $\widetilde{G}$ to be $\pi^{-1}(G)$. Otherwise, $\pi^{-1}(G) \rightarrow G$ splits and there is a subgroup $\widetilde{G}$ of $\pi^{-1}(G)$ such that $\widetilde{G} \cong G$. Then any $G$-invariant curve is defined by a homogeneous polynomial invariant under $\widetilde{G}$.
\end{proof}

Next, we recall the definition of a primitive subgroup of $\PGL(n, \mathbb{C})$.

\begin{definition}\label{def of primitive}
\begin{enumerate}
\item[(1)] Let $G$ be a subgroup of $\GL(n, \mathbb{C})$.

Then $G$ is called an \textit{imprimitive} subgroup of $\GL(n, \mathbb{C})$ if there exists a direct sum decomposition $\mathbb{C}^{n} = V_{1} \oplus \cdots \oplus V_{r}$ with $r > 1$ and $\dim V_{i} > 0$ satisfying  the following: Take any transformation $A \in G$. Then there is a permutation $\sigma \in \mathfrak{S}_{r}$ such that $A \cdot V_{s} = V_{\sigma(s)}$ for any $s$.

Otherwise, $G$ is called a \textit{primitive} subgroup of $\GL(n, \mathbb{C})$.

\item[(2)] Let $G$ be a subgroup of $\PGL(n, \mathbb{C})$. 

Then $G$ is called a \textit{primitive} subgroup of $\PGL(n, \mathbb{C})$ if there exists a lift $\widetilde{G}$ of $G$ which $\widetilde{G}$ is a primitive subgroup of $\GL(n, \mathbb{C})$.
\end{enumerate}
\end{definition}

\begin{remark}\label{rem of primitive}
Let $G$ is a finite primitive subgroup of $\PGL(3, \mathbb{C})$. Then $G$ is conjugate to the Valentiner group $\mathcal{V}$, the icosahedral group $\mathcal{I}$, the Klein group $\mathcal{K}$, the Hessian group $H_{216}$ of order $216$ or its subgroup of order $36$ or $72$. In this groups, $\mathcal{V}$, $\mathcal{K}$ and $H_{216}$ are maximal.

Let $G'$ be a group containing $G$. From the definition, we easily see that $G'$ is also a finite primitive subgroup of $\PGL(3, \mathbb{C})$. Hence, if $G'$ contains $\mathcal{V}$ (resp. $\mathcal{K}$ or $H_{216}$), then $G'$ is equal to $\mathcal{V}$ (resp. $\mathcal{K}$ or $H_{216}$).
\end{remark}

In the following subsections, we recall basic facts on the Valentiner group $\mathcal{V}$, the icosahedral group $\mathcal{I}$ and the Klein group $\mathcal{K}$. Specifically, we give generators of a certain lift $\widetilde{G}$ of $G$ in $\GL(3, \mathbb{C})$ and describe the invariant ring for each of the groups $G = \mathcal{V}$, $\mathcal{I}$ or $\mathcal{K}$.

\subsection{The Valentiner group $\mathcal{V}$}

In this subsection, we give a description of the Valentiner group and its invariant homogeneous polynomials according to \cite{Crass}.

First, we give generators of the Valentiner group. We define the matrices $Z$, $T$, $Q$ and $P$ by
\[
Z :=
\left(
\begin{array}{ccc}
-1 & 0 & 0 \\
0 & 1 & 0 \\
0 & 0 & -1
\end{array}
\right),{\rm {\ }}
T :=
\left(
\begin{array}{ccc}
0 & 0 & 1 \\
1 & 0 & 0 \\
0 & 1 & 0
\end{array}
\right),{\rm {\ }}
\]
\[
Q :=
\left(
\begin{array}{ccc}
1 & 0 & 0 \\
0 & 0 & \rho^{2} \\
0 & -\rho & 0
\end{array}
\right) {\rm and {\ }}
P := \frac{1}{2}
\left(
\begin{array}{ccc}
1 & \tau^{-1} & -\tau \\
\tau^{-1} & \tau & 1 \\
\tau & -1 & \tau^{-1}
\end{array}
\right)
\]
where $\rho = e^{\frac{2}{3} \pi i}$ and $\tau = \displaystyle{\frac{1+\sqrt{5}}{2}}$.

We define $\mathcal{V}$ to be the subgroup of $\PGL(3, \mathbb{C})$ generated by $[Z]$, $[T]$, $[Q]$ and $[P]$. We call this group $\mathcal{V}$ the (projective) \textit{Valentiner group}. It is known that $\mathcal{V}$ is isomorphic to the alternating group $\mathfrak{A}_{6}$.

\begin{remark}\label{icosahedral group}
The matrices $Z$, $T$, $Q$ and $P$ are related to symmetry groups of regular polyhedra. (See [Crass, Subsection 2B].) In particular, the group $\widetilde{\mathcal{I}}$ generated by the matrices $Z$, $T$ and $P$ in $\SL(3, \mathbb{R})$ is the symmetry group  of an icosahedron in $\mathbb{R}^{3}$. It is called the icosahedral group and is isomorphic to the alternating group $\mathfrak{A}_{5}$.

We define $\mathcal{I}$ to be the subgroup of $\PGL(3, \mathbb{C})$ generated by $[Z]$, $[T]$ and $[P]$. Then $\mathcal{I}$ is isomorphic to $\widetilde{\mathcal{I}}$.
\end{remark}

The matrices $Z$, $T$, $Q$ and $P$ are contained in $\SL(3, \mathbb{C})$ and generate the preimage $\widetilde{\mathcal{V}}$ of $\mathcal{V}$ in $\SL(3, \mathbb{C})$. We call this group $\widetilde{\mathcal{V}}$ the \textit{linear Valentiner group}.

Next, we study invariant curves under the Valentiner group.

\begin{lemma}\label{inv of V}
Any projective plane curve invariant under $\mathcal{V}$ is defined by a homogeneous polynomial invariant under $\widetilde{\mathcal{V}}$.
\end{lemma}

\begin{proof}
The alternating group $\mathfrak{A}_{6}$ has no faithful $3$-dimension representation. Thus the natural  homomorphism $\widetilde{\mathcal{V}} \rightarrow \mathcal{V}$ cannot split, and $\widetilde{\mathcal{V}}$ has no nontrivial abelian character. By the arguments of Lemma \ref{semi-invariant prop}, the assertion holds.
\end{proof}

By this lemma, we have only to consider homogeneous polynomials invariant under $\widetilde{\mathcal{V}}$.

Let $P_{\widetilde{\mathcal{V}}}(t)$ be the Poincar\'{e} series of the invariant ring $\mathbb{C}[x, y, z]^{\widetilde{\mathcal{V}}}$. By Molien's theorem, we obtain
\[
P_{\widetilde{\mathcal{V}}}(t) = \frac{1 - t^{90}}{(1 - t^{6})(1 - t^{12})(1 - t^{30})(1 - t^{45})}.
\]
Now, we give $\widetilde{\mathcal{V}}$-invariant and algebraically independent homogeneous polynomials of degrees $6$, $12$ and $30$ and a $\widetilde{\mathcal{V}}$-invariant homogeneous polynomial of degree $45$.

First, we define six quadratic homogeneous polynomials:
\[
C_{\overline{1}}(\bm{x}) := x^{2} + y^{2} + z^{2},
\]
\[
C_{\overline{2}}(\bm{x}) := C_{\overline{1}}(Q^{-1} \bm{x}) = x^{2} + \rho^{2} y^{2} + \rho z^{2},
\]
\[
\begin{split}
C_{\overline{3}}(\bm{x}) & := C_{\overline{2}}(P^{\mathalpha{-}4} \bm{x}) \\
& = \frac{1}{4} \left\{ (\mathalpha{-} 1 \mathalpha{-} \rho \mathalpha{+} \tau \mathalpha{+} 2\rho\tau) x^{2}
+ (1 \mathalpha{-} 2\tau \mathalpha{-} \rho\tau) y^{2}
+ (\rho \mathalpha{+} \tau \mathalpha{-} \rho\tau) z^{2} \right. \\
& \quad \left. \mathalpha{+} (\mathalpha{-} 4 \mathalpha{-} 2\rho \mathalpha{+} 2\tau \mathalpha{-}\rho\tau) x y
+ (\mathalpha{-} 2 \mathalpha{+} 2\rho \mathalpha{-} 2\tau \mathalpha{-} 4\rho\tau) y z
+ (2 \mathalpha{+} 4\rho \mathalpha{-} 4\tau \mathalpha{-} 2\rho\tau) z x \right\}, \\
\end{split}
\]
\[
\begin{split}
C_{\overline{4}}(\bm{x}) & := C_{\overline{2}}(P^{\mathalpha{-}3} \bm{x}) \\
& = \frac{1}{4} \left\{ (1 \mathalpha{-} 2\tau \mathalpha{-} \rho\tau) x^{2}
+ (\rho \mathalpha{+} \tau \mathalpha{-} \rho\tau) y^{2}
+ (\mathalpha{-} 1 \mathalpha{-} \rho \mathalpha{+} \tau \mathalpha{+} 2\rho\tau) z^{2} \right. \\
& \quad \left. \mathalpha{+} (\mathalpha{-} 2 \mathalpha{+} 2\rho \mathalpha{-} 2\tau \mathalpha{-} 4\rho\tau) x y
+ (2 \mathalpha{+} 4\rho \mathalpha{-} 4\tau \mathalpha{-} 2\rho\tau) y z
+ (\mathalpha{-} 4 \mathalpha{-} 2\rho \mathalpha{+} 2\tau \mathalpha{-} 2\rho\tau) z x \right\}, \\
\end{split}
\]
\[
\begin{split}
C_{\overline{5}}(\bm{x}) & := C_{\overline{2}}(P^{\mathalpha{-}2} \bm{x}) \\
& = \frac{1}{4} \left\{ (1 \mathalpha{-} 2\tau \mathalpha{-} \rho\tau) x^{2}
+ (\rho \mathalpha{+} \tau \mathalpha{-} \rho\tau) y^{2}
+ (\mathalpha{-} 1 \mathalpha{-} \rho \mathalpha{+} \tau \mathalpha{+} 2\rho\tau) z^{2} \right. \\
& \quad \left. \mathalpha{+} (\mathalpha{-} 2 \mathalpha{+} 2\rho \mathalpha{-} 2\tau \mathalpha{-} 4\rho\tau) x y
+ (\mathalpha{-} 2 \mathalpha{-} 4\rho \mathalpha{+} 4\tau \mathalpha{+} 2\rho\tau) y z
+ (4 \mathalpha{+} 2\rho \mathalpha{-} 2\tau \mathalpha{+} 2\rho\tau) z x \right\} \\
\end{split}
\]
and
\[
\begin{split}
C_{\overline{6}}(\bm{x}) & := C_{\overline{2}}(P^{\mathalpha{-}1} \bm{x}) \\
& = \frac{1}{4} \left\{ (\mathalpha{-} 1 \mathalpha{-} \rho \mathalpha{+} \tau \mathalpha{+} 2\rho\tau) x^{2}
+ (1 \mathalpha{-} 2\tau \mathalpha{-} \rho\tau) y^{2}
+ (\rho \mathalpha{+} \tau \mathalpha{-} \rho\tau) z^{2} \right. \\
& \quad \left. \mathalpha{+} (\mathalpha{-} 4 \mathalpha{-} 2\rho \mathalpha{+} 2\tau \mathalpha{-}\rho\tau) x y
+ (2 \mathalpha{-} 2\rho \mathalpha{+} 2\tau \mathalpha{+} 4\rho\tau) y z
+ (\mathalpha{-} 2 \mathalpha{-} 4\rho \mathalpha{+} 4\tau \mathalpha{+} 2\rho\tau) z x \right\}. \\
\end{split}
\]
Since $Z$, $T$ and $P$ are orthogonal matrices, $C_{\overline{1}}(x, y, z)$ is invariant under $\mathcal{I}$. We have $(\mathcal{V} : \mathcal{I}) = 6$, and it turn out that, for any $A \in \mathcal{\widetilde{V}}$, there exist indices $k$ and $j = 1, \cdots, 6 $ such that $C_{\overline{i}}^{A} = \rho^{k} C_{\overline{j}}$, and the Valentiner group $\mathcal{V}$ permutes the six conics.

The Valentiner group $\mathcal{V}$ fixes the sextic homogeneous polynomial
\[
F_{\mathcal{V}}(x, y, z) := \sum^{6}_{n = 1} \left( C_{\overline{n}}(x, y, z) \right)^{3}.
\]
This is a homogeneous polynomial of degree $6$ invariant under $\widetilde{\mathcal{V}}$. (See Appendix A.1 for the explicit form in the ``Wiman coordinate''.) We define the homogeneous polynomial $\Phi_{\mathcal{V}}(x, y, z)$ of degree $12$ as the Hessian of $F_{\mathcal{V}}$, i.e., 
\[
\Phi_{\mathcal{V}}(x, y, z) := \det H(F_{\mathcal{V}})
= \left|
\begin{array}{ccc}
\dfrac{\partial^{2} F_{\mathcal{V}}}{\partial x^{2}} & \dfrac{\partial^{2} F_{\mathcal{V}}}{\partial x \partial y} & \dfrac{\partial^{2} F_{\mathcal{V}}}{\partial x \partial z} \\
&&\\
\dfrac{\partial^{2} F_{\mathcal{V}}}{\partial y \partial x} & \dfrac{\partial^{2} F_{\mathcal{V}}}{\partial y^{2}} & \dfrac{\partial^{2} F_{\mathcal{V}}}{\partial y \partial z}\\
&&\\
\dfrac{\partial^{2} F_{\mathcal{V}}}{\partial z \partial x} & \dfrac{\partial^{2} F_{\mathcal{V}}}{\partial z \partial y} & \dfrac{\partial^{2} F_{\mathcal{V}}}{\partial z^{2}}
\end{array}
\right|.
\]
Furthermore, the homogeneous polynomial $\Psi_{\mathcal{V}}(x, y, z)$ of degree $30$ is defined as the border Hessian of $F_{\mathcal{V}}$ and $\Phi_{\mathcal{V}}$, i.e.,
\[
\Psi_{\mathcal{V}}(x, y, z) := \det BH(F_{\mathcal{V}}, \Phi_{\mathcal{V}})
= \left|
\begin{array}{ccc|c}
& & & \dfrac{\partial \Phi_{\mathcal{V}}}{\partial x} \\
&&&\\
& H(F_{\mathcal{V}}) & & \dfrac{\partial \Phi_{\mathcal{V}}}{\partial y} \\
&&&\\
& & & \dfrac{\partial \Phi_{\mathcal{V}}}{\partial z} \\
&&&\\
\hline
&&&\\
\dfrac{\partial \Phi_{\mathcal{V}}}{\partial x} & \dfrac{\partial \Phi_{\mathcal{V}}}{\partial y} & \dfrac{\partial \Phi_{\mathcal{V}}}{\partial z} & 0
\end{array}
\right|.
\]
The homogeneous polynomial $X_{\mathcal{V}}(x, y, z)$ of degree $45$ is defined as the Jacobian of $(F_{\mathcal{V}}, \Phi_{\mathcal{V}}, \Psi_{\mathcal{V}})$, i.e.,
\[
X_{\mathcal{V}}(x, y, z) := \det J(F_{\mathcal{V}}, \Phi_{\mathcal{V}}, \Psi_{\mathcal{V}})
= \left|
\begin{array}{ccc}
\dfrac{\partial F_{\mathcal{V}}}{\partial x} & \dfrac{\partial F_{\mathcal{V}}}{\partial y} & \dfrac{\partial F_{\mathcal{V}}}{\partial z} \\
&&\\
\dfrac{\partial \Phi_{\mathcal{V}}}{\partial x} & \dfrac{\partial \Phi_{\mathcal{V}}}{\partial y} & \dfrac{\partial \Phi_{\mathcal{V}}}{\partial z}\\
&&\\
\dfrac{\partial \Psi_{\mathcal{V}}}{\partial x} & \dfrac{\partial \Psi_{\mathcal{V}}}{\partial y} & \dfrac{\partial \Psi_{\mathcal{V}}}{\partial z}
\end{array}
\right|.
\]
It is easy to see that $\Phi_{\mathcal{V}}$, $\Psi_{\mathcal{V}}$ and $X_{\mathcal{V}}$ are invariant under $\widetilde{\mathcal{V}}$.

\subsection{The icosahedral group $\mathcal{I}$}

In this subsection, we give the invariant ring for the icosahedral group $\mathcal{I}$.

First, from Remark \ref{icosahedral group}, we recall the definition of the icosahedral group $\widetilde{\mathcal{I}}$, isomorphic to $\mathfrak{A}_{5}$ and that $\mathcal{I}$ is the isomorphic image of $\widetilde{\mathcal{I}}$ in $\PGL(3, \mathbb{C})$.

Next, we look at invariant curves under $\mathcal{I}$. It suffices to consider invariant homogeneous polynomials under $\widetilde{\mathcal{I}}$ by the following lemma.

\begin{lemma}\label{inv of I}
Any projective plane curve invariant under $\mathcal{I}$ is defined by a homogeneous polynomial invariant under $\widetilde{\mathcal{I}}$.
\end{lemma}

\begin{proof}
This follows from the arguments of Lemma \ref{semi-invariant prop} and the fact that $\widetilde{\mathcal{I}} \cong \mathcal{I} \cong \mathfrak{A}_{5}$ is simple and nonabelian.
\end{proof}
The Poincar\'{e} series of $\widetilde{\mathcal{I}}$ is
\[
P_{\widetilde{\mathcal{I}}}(t) = \frac{1 - t^{30}}{(1 - t^{2})(1 - t^{6})(1 - t^{10})(1 - t^{15})}.
\]
We explicitly give $\widetilde{\mathcal{I}}$-invariant and algebraically independent homogeneous polynomials of degree $2$, $6$ and $10$ and a $\widetilde{\mathcal{I}}$-invariant homogeneous polynomial of degree $15$.

As we saw in the previous subsection, the homogeneous polynomial $C_{\overline{1}}$ is invariant under $\mathcal{I} \subset \mathcal{V}$. We define the quadratic homogeneous polynomial
\[
F_{\mathcal{I}}(x, y, z) := C_{\overline{1}}(x, y, z) = x^{2} + y^{2} + z^{2}.
\]
In addition, $F_{\mathcal{V}}$ is invariant under $\mathcal{V}$, hence under $\mathcal{I}$. We define the sextic homogeneous polynomial
\[
\Phi_{\mathcal{I}}(x, y, z) := F_{\mathcal{V}}(x, y, z).
\]
As in the case of the Valentiner group, the homogeneous polynomial $\Psi_{\mathcal{I}}$ of degree $10$ is defined as the border Hessian of $F_{\mathcal{I}}$ and $\Phi_{\mathcal{I}}$, i.e.,
\[
\Psi_{\mathcal{I}}(x, y, z) := \det BH(F_{\mathcal{I}}, \Phi_{\mathcal{I}})
\]
and the homogeneous polynomial $X_{\mathcal{I}}$ of degree $15$ is defined as the Jacobian of $(F_{\mathcal{I}}, \Phi_{\mathcal{I}}, \Psi_{\mathcal{I}})$, i.e.,
\[
X_{\mathcal{I}}(x, y, z) := \det J(F_{\mathcal{I}}, \Phi_{\mathcal{I}}, \Psi_{\mathcal{I}}).
\]
Then $\Psi_{\mathcal{I}}$ and $X_{\mathcal{I}}$ are invariant under $\widetilde{\mathcal{I}}$.

\subsection{The Klein group $\mathcal{K}$}

In this subsection, we describe an embedding $\mathcal{K}$ of the Klein group $\PSL(2, \mathbb{F}_{7})$ to $\PGL(3, \mathbb{C})$ and describe homogeneous polynomials invariant under $\mathcal{K}$ following \cite{Elkies}.

The group $\PSL(2, \mathbb{F}_{7})$ has a faithful $3$-dimensional representation $\PSL(2, \mathbb{F}_{7}) \rightarrow \GL(3, \mathbb{C})$. By [Elkies, Subsection 1.1], we may take an embedding for which the image $\widetilde{\mathcal{K}}$ is generated by three matrices
\[
S :=
\left(
\begin{array}{ccc}
\zeta^{4} & 0 & 0 \\
0 & \zeta^{2} & 0 \\
0 & 0 & \zeta
\end{array}
\right),{\rm {\ }}
T :=
\left(
\begin{array}{ccc}
0 & 0 & 1\\
1 & 0 & 0\\
0 & 1 & 0\\
\end{array}
\right)
\]
and
\[
R :=
-\frac{1}{\sqrt{-7}}
\left(
\begin{array}{ccc}
\zeta-\zeta^{6} & \zeta^{2}-\zeta^{5} & \zeta^{4}-\zeta^{3} \\
\zeta^{2}-\zeta^{5} & \zeta^{4}-\zeta^{3} & \zeta-\zeta^{6} \\
\zeta^{4}-\zeta^{3} & \zeta-\zeta^{6} & \zeta^{2}-\zeta^{5}
\end{array}
\right)
\]
where $\zeta := e^{\frac{2\pi i}{7}}$. Let $\mathcal{K}$ be the image of $\widetilde{\mathcal{K}}$ by the projection $\GL(3, \mathbb{C}) \rightarrow \PGL(3, \mathbb{C})$. Then $\mathcal{K}$ is isomorphic to $\widetilde{\mathcal{K}}$.

Next, we study invariant curves under $\mathcal{K}$. We see the following lemma in the same way of Lemma \ref{inv of I}.

\begin{lemma}\label{inv of K}
Any projective plane curve invariant under $\mathcal{K}$ is defined by a homogeneous polynomial invariant under $\widetilde{\mathcal{K}}$.
\end{lemma}

The Poincar\'{e} series $P_{\widetilde{\mathcal{K}}}(t)$ is given by
\[
P_{\widetilde{\mathcal{K}}}(t) = \frac{1-t^{42}}{(1-t^{4})(1-t^{6})(1-t^{14})(1-t^{21})}.
\]
We explicitly give $\widetilde{\mathcal{K}}$-invariant and algebraically independent homogeneous polynomials of degrees $4$, $6$ and $14$ and a $\widetilde{\mathcal{K}}$-invariant homogeneous polynomial of degree $21$.

The homogeneous polynomial $F_{\mathcal{K}}$ of degree $4$ is defined by
\[
F_{\mathcal{K}}(x, y, z) := x^{3} y + y^{3} z + z^{3} x.
\]

\begin{remark}
The curve $K_{4}$ defined by $F_{\mathcal{K}}$ is called the \textit{Klein quartic}. The automorphism group of $K_{4}$ is isomorphic to $\PSL(2, \mathbb{F}_{7})$, which has been known since \cite{Klein}. The group $\PSL(2, \mathbb{F}_{7}) \cong \Aut V(K_{4})$ is called the \textit{Klein group}.
\end{remark}

We can give other invariant homogeneous polynomials as in the case of the Valentiner group. We define the homogeneous polynomials $\Phi_{\mathcal{K}}$, $\Psi_{\mathcal{K}}$ and $X_{\mathcal{K}}$ of degree $6$, $14$ and $21$, respectively, by
\[
\Phi_{\mathcal{K}}(x, y, z) := -\frac{1}{54} \det H(F_{\mathcal{K}}),
\]
\[
\Psi_{\mathcal{K}}(x, y, z) := -\frac{1}{9} \det BH(F_{\mathcal{K}}, \Phi_{\mathcal{K}})
\]
and
\[
X_{\mathcal{K}}(x, y, z) := \det J(F_{\mathcal{K}}, \Phi_{\mathcal{K}}, \Psi_{\mathcal{K}}).
\]
One can see that $F_{\mathcal{K}}$, $\Phi_{\mathcal{K}}$, $\Psi_{\mathcal{K}}$ and $X_{\mathcal{K}}$ are invariant under $\widetilde{\mathcal{K}}$.

\subsection{The invariant curves}

In this subsection, we summarize some facts related to invariant homogeneous polynomials in a form that is convenient for later use.

Let $G$ be the Valentiner group $\mathcal{V}$, the icosahedral group $\mathcal{I}$ or the Klein group $\mathcal{K}$. If $G = \mathcal{V}$, for example, then $F_{G}$ denotes the $\mathcal{V}$-invariant homogeneous polynomial $F_{\mathcal{V}}$, and so on. Take $\widetilde{G}$ to be the corresponding lift $\widetilde{\mathcal{V}}$, $\widetilde{\mathcal{I}}$ or $\widetilde{\mathcal{K}}$.

First, we tabulate the order of $G$ and the degrees of the homogeneous polynomials $F_{G}$, $\Phi_{G}$, $\Psi_{G}$ and $X_{G}$ for each group $G$.
\begin{table}[h]
\caption{degrees of invariants}
\[
\begin{array}{|c|c|cccc|}
\hline 
G & \# G & \deg F_{G} & \deg \Phi_{G} & \deg \Psi_{G} & \deg X_{G} \\
\hline
\mathcal{V} & 360 & 6 & 12 & 30 & 45 \\
\mathcal{I} & 60 & 2 & 6 & 10 & 15 \\
\mathcal{K} & 168 & 4 & 6 & 14 & 21 \\
\hline
\end{array}
\]
\label{tab:degree}
\end{table}

By Lemmas \ref{inv of V}, \ref{inv of I} and \ref{inv of K}, we obtain the following.

\begin{lemma}\label{inv of G}
Any projective plane curve invariant under $G$ is defined by a homogeneous polynomial invariant under $\widetilde{G}$.
\end{lemma}

Next, we have the following lemma on the invariant ring and invariant curves for each group.

\begin{lemma}\label{basic prop of invariants}
For each of groups $G = \mathcal{V}, \mathcal{I}$ and $\mathcal{K}$, the homogeneous polynomials $F_{G}$, $\Phi_{G}$, $\Psi_{G}$ and $X_{G}$ satisfy the following:
\begin{enumerate}
\item[(0)] $\mathbb{C}[x, y, z]^{\widetilde{G}} = \mathbb{C}[F_{G}, \Phi_{G}, \Phi_{G}, X_{G}]$.
\item[(1)] The homogeneous polynomials $F_{G}, \Phi_{G}$ and $\Psi_{G}$ are algebraically independent over $\mathbb{C}$.
\item[(2)] The homogeneous polynomial $X_{G}^{2}$ is in $\mathbb{C}[F_{G}, \Phi_{G}, \Psi_{G}]$.
\item[(3)] The curve $V(X_{G})$ is a union of $\deg X_{G}$ lines. Thus it is reducible.
\item[(4)] $V(F_{G}) \cap V(\Phi_{G}) \cap V(\Psi_{G}) = \emptyset$.
\item[(5)] The curves $V(F_{G})$, $V(\Phi_{G})$ and $V(\Psi_{G})$ are nonsingular.
\item[(6)] The curves $V(F_{G})$ and $V(\Phi_{G})$ meet transversally. 
\item[(7)] The sets $V(F_{G}) \cap V(\Phi_{G})$ and $V(F_{G}) \cap V(\Psi_{G})$ are $G$-orbits.
\end{enumerate}
\end{lemma}

\begin{proof}
First, we have the claims (0) to (3) for $G = \mathcal{V}$ by [Crass, Subsection 2E] and for $G = \mathcal{K}$ by [Elkies, Subsection 1.2].

We show the claims (0) to (3) for $G = \mathcal{I}$.

The claim (1) follows from the fact the Jacobian of $(F_{\mathcal{I}}, \Phi_{\mathcal{I}}, \Psi_{\mathcal{I}})$ (i.e. $X_{\mathcal{I}}$) is nonzero. We checked the latter using the computer algebra system SINGULAR. (See Appendix A.2.) Moreover, from the Poincar\'{e} series, we also see the claim (0).

Any invariant homogeneous polynomial under $\pm\widetilde{\mathcal{I}}$ is contained in $\mathbb{C}[F_{\mathcal{I}}, \Phi_{\mathcal{I}}, \Psi_{\mathcal{I}}]$ by [ST, 5.1]. Since $X_{\mathcal{I}}^{2}$ is invariant under $\pm\widetilde{\mathcal{I}}$, we obtain the claim (2).

Next, we prove the claim (3). There are $15$ elements of order $2$ in $\mathcal{I}$. These elements fix pairwise distinct lines in $\mathbb{P}^{2}$. Let $C$ be the union of these $15$ lines. Then the curve $C$ is invariant under $\mathcal{I}$. On the other hand, since $F_{\mathcal{I}}$, $\Phi_{\mathcal{I}}$ and $\Psi_{\mathcal{I}}$ are of even degrees, any invariant homogeneous polynomial of degree $15$ under $\widetilde{\mathcal{I}}$ is $\alpha X_{\mathcal{I}}$ with $\alpha \in \mathbb{C}$. Therefore, the curve $V(X_{\mathcal{I}})$ is equal to $C$.

We can check the claims (4) to (6) using SINGULAR. (See Appendix A.1 to A.3.) For the claim (4), we calculate the radical of the ideal generated by $F_{G}$, $\Phi_{G}$ and $\Psi_{G}$. Then we see that it is equal to the ideal $\left< x, y, z \right>$ generated by $x$, $y$ and $z$. Hence, the equation $F_{G} = \Phi_{G} = \Psi_{G} = 0$ has only the trivial solution $(x, y, z) = (0, 0, 0)$.

For the claim (5), we can check that the radical of the ideal generated by the derivatives $\left(\dfrac{\partial F_{G}}{\partial x}\right)$, $\left(\dfrac{\partial F_{G}}{\partial y}\right)$ and $\left(\dfrac{\partial F_{G}}{\partial z}\right)$ is equal to the ideal $\left< x, y, z \right>$. Thus $V(F_{G})$ is nonsingular. In the same way, we can check that $V(\Phi_{G})$ and $V(\Psi_{G})$ are nonsingular.

For the claim (6), we compare the ideal generated by $F_{G}$ and $\Phi_{G}$ with its radical. We can check that two ideals are equal. Thus $V(F_{G})$ and $V(\Phi_{G})$ meet transversally.

Finally, we prove the claim (7). Let $C = V(F_{G})$, $D_{1} = V(\Phi_{G})$ and $D_{2} = V(\Psi_{G})$. Since the curves $C$, $D_{1}$ and $D_{2}$ are invariant under $G$, $G$ acts on each of the set $C \cap D_{1}$ and $C \cap D_{2}$. Since the set $(C \cap D_{1}) \cap (C \cap D_{2})$ is empty by (4), we may write $C \cap D_{1} = \coprod\limits_{i = 1}^{m} O_{i}$ and $C \cap D_{2} = \coprod\limits_{i = m + 1}^{m + n} O_{i}$ where $O_{1}, \cdots, O_{m + n}$ are distinct $G$-orbits. We can take a natural morphism $C \rightarrow C/G$, and recall the Hurwitz's theorem. Then the ramification index at $P \in O_{i}$ is $\left(\dfrac{\# G}{\# O_{i}} - 1\right)$ and we see
\[
\begin{array}{ccc}
\sum\limits_{i = 1}^{m} \left\{ \# O_{i} \cdot \left( \dfrac{\# G}{\# O_{i}} - 1\right) \right\} & = & m \# G - \# (C \cap D_{1}) \\
\end{array}
\]
and
\[
\begin{array}{ccc}
\sum\limits_{i = m+1}^{m+n} \left\{ \# O_{i} \cdot \left( \dfrac{\# G}{\# O_{i}} - 1\right) \right\} & = & n \# G - \# (C \cap D_{2}). \\
\end{array}
\]
Hence, we obtain the inequality
\[
\begin{array}{ccl}
2g(C) - 2 & \geq & \# G \cdot (2g(C/G) - 2)\\
&&\\
&& + (m\# G - \#(C \cap D_{1})) + (n\# G - \#(C \cap D_{2})) \\
&&\\
& \geq & (m + n - 2) \# G - \deg F_{G}(\deg \Phi_{G} + \deg \Psi_{G})
\end{array}
\]
since $g(C/G) \geq 0$, $\#(C \cap D_{1}) \leq \deg F_{G} \deg \Phi_{G}$ and $\#(C \cap D_{2}) \leq \deg F_{G} \deg \Psi_{G}$. For each group $G = \mathcal{V}$, $\mathcal{I}$ or $\mathcal{K}$, the genus $g(C)$ is $10$, $0$ or $3$.  If $m + n - 2 \geq 1$, then it is straightforward to see that the inequality does not hold. Therefore, $m = n = 1$, i.e., $C \cap D_{1}$ and $C \cap D_{2}$ are $G$-orbits.
\end{proof}

\begin{remark}
By [Crass, Subsection 4B], the set $V(F_{\mathcal{V}}) \cap V(\Phi_{\mathcal{V}})$ is a $\mathcal{V}$-orbit $\mathcal{O}_{72}$ of order $72$ and the set $V(F_{\mathcal{V}}) \cap V(\Psi_{\mathcal{V}})$ is a $\mathcal{V}$-orbit $\mathcal{O}_{90}$ of order $90$, which proves the claims (4), (6) and (7) for $G = \mathcal{V}$.
\end{remark}

Furthermore, we can drop $X_{G}$ in considering \emph{irreducible} invariant curves.

\begin{prop}\label{generators of irreducible invariants}
If $C$ is an irreducible curve invariant under $G$, then $\deg C$ is even. In particular, any integral curve invariant under $G$ is defined by a homogeneous polynomial in $\mathbb{C}[F_{G}, \Phi_{G}, \Psi_{G}]$.
\end{prop}

\begin{proof}
Take a curve $C$ of degree $d$ invariant under $G$. By Lemma \ref{inv of G}, $C$ is defined by a homogeneous polynomial $H$ of degree $d$ invariant under $\widetilde{G}$. By Lemma \ref{basic prop of invariants} (0), $H$ is contained in $\mathbb{C}[F_{G}, \Phi_{G}, \Psi_{G}, X_{G}]$. We note that the degrees of $F_{G}$, $\Phi_{G}$ and $\Psi_{G}$ are even and the degree of $X_{G}$ is odd. Assume that $d$ is odd. Then $H$ is divisible by $X_{G}$. By Lemma \ref{basic prop of invariants} (3), the curve $V(H)$ is reducible. Hence, if $V(H)$ is irreducible, then $d$ is even.

Finally, suppose that $C = V(H)$ is integral. Since the degree of $H$ is even, $H$ is contained in $\mathbb{C}[F_{G}, \Phi_{G}, \Psi_{G}, X_{G}^{2}] = \mathbb{C}[F_{G}, \Phi_{G}, \Psi_{G}]$ by Lemma \ref{basic prop of invariants} (2).
\end{proof}

\section{Nonsingular curves whose automorphism groups are simple and primitive}

In this section, we consider nonsingular curves invariant under $G = \mathcal{V}$, $\mathcal{I}$ or $\mathcal{K}$. By Proposition \ref{generators of irreducible invariants}, any such curve is defined by a homogeneous polynomial in $\mathbb{C}[F_{G}, \Phi_{G}, \Psi_{G}]$.

The space of homogeneous polynomials of degree $d$ corresponds to the complete linear system of degree $d$ curves. Since a $G$-invariant curve is defined by a homogeneous polynomial invariant under $\widetilde{G}$ by Lemma \ref{semi-invariant prop}, the set of $G$-invariant curves of degree $d$ is a linear system. We denote this linear system by $(\mathfrak{d}_{G})_{d}$.

When a group $G$ is fixed, we put $a = \deg F_{G}$, $b = \deg \Phi_{G}$ and $c = \deg \Psi_{G}$. For any group $G$, we see that $a < b < c$.

\begin{remark}
By Proposition \ref{generators of irreducible invariants}, it suffices to consider the case of an even degree $d$. In this case, the linear system $(\mathfrak{d}_{G})_{d}$ corresponds to the linear space generated by $F^{i} \Phi^{j} \Psi^{k}$ with $d = ai + bj + ck$. We often identify an element of $(\mathfrak{d}_{G})_{d}$ with a homogeneous polynomial defining the curve, which can be written as a linear combination of $F^{i} \Phi^{j} \Psi^{k}$ with $d = ai + bj + ck$.
\end{remark}

In this section, we find a necessary and sufficient condition on $d$ for the existence of nonsingular elements of $(\mathfrak{d}_{G})_{d}$. First, we translate the condition to an arithmetical one.

\begin{prop}\label{condition of nonsingularity}
Let $d \geq c$. There is a nonsingular element of $(\mathfrak{d}_{G})_{d}$ if and only if all of the following hold:
\begin{enumerate}
\item[(1)] There is a pair of nonnegative integers  $(j, k)$ such that $d = bj + ck$.
\item[(2)] There is a pair of nonnegative integers  $(i, k)$ such that $d = ai + ck$.
\item[(3)] There is a pair of nonnegative integers  $(i, j)$ such that $d = ai + bj$.
\item[(4)] $d \equiv 0, a$ or $ b \mod c$.
\item[(5)] $d \equiv 0, a$ or $c \mod b$.
\item[(6)] $d \equiv 0, b$ or $c \mod a$.
\end{enumerate}
If these conditions are satisfied, then a general member of $(\mathfrak{d}_{G})_{d}$ is nonsingular.
\end{prop}

To prove this proposition, we show that each of the conditions (1) to (6) is equivalent to a certain condition on the singularity of a general element of the linear system.

\begin{note}
In this section, when $G$ is fixed, we define $F_{1} := F_{G}$, $F_{2} := \Phi_{G}$ and $F_{3} := \Psi_{G}$.  Then take $a_{n} := \deg F_{n}$ for $n = 1, 2$ and $3$, i.e., $a_{1} = a$, $a_{2} = b$ and $a_{3} = c$.
\end{note}

\begin{lemma}\label{reducible condition}
Let $d$ be a positive integer and take indices $l$, $m$ and $n$ such that $\{ l, m, n\} = \{1, 2, 3\}$. There exists an element of $(\mathfrak{d}_{G})_{d}$ which is not divisible by $F_{l}$ if and only if there is a pair of nonnegative integers $(s, t)$ such that $d = a_{m} s + a_{n} t$.
\end{lemma}

\begin{proof}
First, we clearly see the assumption if $(\mathfrak{d}_{G})_{d}$ is empty.

Assume that there exists an element of $(\mathfrak{d}_{G})_{d}$ which is not divisible by $F_{l}$. Then $F_{m}^{s}F_{n}^{t}F_{l}^{0} = F_{m}^{s}F_{n}^{t}$ belongs to $(\mathfrak{d}_{G})_{d}$ for some $s, t \geq 0$. Thus there is a pair of nonnegative integers $(s, t)$ with $a_{m}s + a_{n}t = d$. Conversely, if a pair of nonnegative integers $(s_{0}, t_{0})$ satisfies $d = a_{m} s_{0} + a_{n} t_{0}$, then $F_{m}^{s_{0}} F_{n}^{t_{0}}$ belongs to $(\mathfrak{d}_{G})_{d}$, and is not divisible by $F_{l}$.
\end{proof}

\begin{lemma}\label{intersection base pts condition}
Take indices $l, m$ and $n$ such that $\{ l, m, n\} = \{ 1, 2, 3\}$ and let $d \geq \max \{a_{m}, a_{n}\}$ be an integer. Then the linear system $(\mathfrak{d}_{G})_{d}$ has an element of the form $F_{m}^{s} F_{n}^{t} F_{l}^{u}$ satisfying $s+t \leq 1$ if and only if $d \equiv 0$, $a_{m}$ or $a_{n} \mod a_{l}$.
\end{lemma}

\begin{proof}
First, assume that an element of the form $F_{m}^{s} F_{n}^{t} F_{l}^{u}$ satisfying $s+t \leq 1$ is in $(\mathfrak{d}_{G})_{d}$. Since $a_{m}s + a_{n}t + a_{l}u = d$, we see that $d \equiv a_{m}s + a_{n}t \mod a_{l}$. Therefore, $d \equiv 0$, $a_{m}$ or $a_{n} \mod a_{l}$

Conversely, if $d$ is divisible by $a_{l}$, then $(\mathfrak{d}_{G})_{d}$ clearly has an element of the form $F_{l}^{\frac{d}{a_{l}}}$. If $d$ satisfies the condition (i) or (ii), then there is an integer $u$ such that $d = a_{l} u + a_{m}$ or $d = a_{l} u + a_{n}$ for each case, and $u \geq 0$ since $d \geq \max\{a_{m}, a_{n}\}$. Thus an element of the form $F_{m} F_{l}^{u}$ or $F_{n} F_{l}^{u}$ is contained in $(\mathfrak{d}_{G})_{d}$ for such $u$.
\end{proof}

By Lemma \ref{reducible condition}, we can translate the conditions (1), (2) and (3) into conditions on the base points of $(\mathfrak{d}_{G})_{d}$. Moreover, by Lemma \ref{intersection base pts condition}, we can also translate the conditions (4), (5) and (6). Now we can prove Proposition \ref{condition of nonsingularity}.

\begin{proof}(The proof of Proposition \ref{condition of nonsingularity}.)

First, we assume that one of the conditions (1) to (6) does not hold. Suppose that $(\mathfrak{d}_{G})_{d}$ is nonempty and show that any element of $(\mathfrak{d}_{G})_{d}$ is singular. (If it is empty, then there is no nonsingular element.)

If $d = c$, then it is easy to see the conditions (1) to (3) by calculation for each group. Suppose that $d > c$. By Lemma \ref{reducible condition}, if $d$ does not satisfy one of the conditions (1), (2) and (3), then any element of $(\mathfrak{d}_{G})_{d}$ is divisible by $F_{G}, \Phi_{G}$ or $\Psi_{G}$. Since $d > c$, any element of $(\mathfrak{d}_{G})_{d}$ is reducible or nonreduced, hence it's singular.

Assume that one of the conditions (4), (5) and (6) does not hold. Then we can take a set of indices $\{l, m, n\} = \{1, 2, 3\}$ such that $d$ does not satisfy $d \equiv 0, a_{m}$ or $a_{n} \mod a_{l}$. Let $x_{1} = x$, $x_{2} = y$ and $x_{3} = z$. The homogeneous polynomial $H$ is a linear combination of
\[
T_{s, t, u} := F_{l}^{s} F_{m}^{t} F_{n}^{u}
\]
with $d = a_{l} s + a_{m} t + a_{n} u$. The derivative of $T_{s, t, u}$ is
\[
\dfrac{\partial}{\partial x_{r}} T_{s, t, u} = s \left(\dfrac{\partial F_{l}}{\partial x_{r}}\right) T_{s-1, t, u} + t \left(\dfrac{\partial F_{m}}{\partial x_{r}}\right) T_{s, t-1, u} + u \left(\dfrac{\partial F_{n}}{\partial x_{r}}\right) T_{s, t, u-1}
\]
for $r = 1, 2$ and $3$. By the assumption on $d$ and Lemma \ref{intersection base pts condition}, $H$ is a linear combination of  $T_{s, t, u}$ with $t + u \geq 2$. It is easy to see that $H$ and its derivatives $\displaystyle{\frac{\partial}{\partial x_{r}}} H$ vanishes on $V(F_{m}) \cap V(F_{n})$. Therefore, the curve $V(H)$ is singular at any point in $V(F_{m}) \cap V(F_{n})$.

Now we assume that all of the conditions (1) to (6) hold and show that a general element of $(\mathfrak{d}_{G})_{d}$ is nonsingular. By Bertini's theorem, a general member of the linear system $(\mathfrak{d}_{G})_{d}$ can have singular points only at the base points. We will prove that it is also nonsingular there.

By the conditions (1) to (3) and Lemma \ref{reducible condition}, the base locus $\Lambda$ satisfies
\[
\begin{array}{ccc}
\Lambda & \subset & (V(\Phi_{G}) \cup V(\Psi_{G})) \cap (V(F_{G}) \cup V(\Psi_{G})) \cap (V(F_{G}) \cup V(\Phi_{G})) \\
& = & (V(F_{G}) \cap V(\Phi_{G})) \cup (V(F_{G}) \cap V(\Psi_{G})) \cup (V(\Phi_{G}) \cap V(\Psi_{G})).
\end{array}
\]
Thus take any pair $\{m, n\}$ and any point $P \in V(F_{m}) \cap V(F_{n})$, and we show that a general element $H$ of $(\mathfrak{d}_{G})_{d}$ either does not certain $P$ or is nonsingular at $P$.

Let $l$ be an index satisfying $\{l, m, n\} = \{1, 2, 3\}$. By Lemma \ref{intersection base pts condition} and the conditions (4), (5) and (6), $(\mathfrak{d}_{G})_{d}$ has an element of the form $E := F_{l}^{s}F_{m}^{t}F_{n}^{u}$ with $t + u \leq 1$. If $t + u = 0$, i.e., an element of the form $F_{l}^{\frac{d}{a_{l}}}$ is contained in $(\mathfrak{d}_{G})_{d}$, then a general member is nonsingular does not pass through $P$ since $V(F_{G}) \cap V(\Phi_{G}) \cap V(\Psi_{G}) = \emptyset$.

Suppose that $t + u = 1$. Let $p$ be $m$ or $n$ according as $t = 1$ or $u = 1$. The derivative of $E$ is
\[
\frac{\partial}{\partial x_{r}} E = \left(\frac{\partial F_{p}}{\partial x_{r}}\right) F_{l}^{\frac{d-a_{p}}{a_{l}}} + \frac{d - a_{p}}{a_{l}} \left(\frac{\partial F_{l}}{\partial x_{r}}\right) F_{p}F_{l}^{\frac{d-a_{p}}{a_{l}}-1}
\]
for $r = 1, 2$ and $3$. By Lemma \ref{basic prop of invariants} (4), $F_{l}(P)$ is nonzero. Moreover, since $V(F_{p})$ is nonsingular by Lemma \ref{basic prop of invariants} (5), there is an index $r$ such that $\dfrac{\partial F_{p}}{\partial x_{r}}$ is nonzero at $P$. Hence, the derivative $\displaystyle{\frac{\partial}{\partial x_{r}}} E$ does not vanish at $P$ for an index $r$. Therefore, a general element is nonsingular at $P$.
\end{proof}

From the proof, we also see the next lemma, which will be used in Section 4.

\begin{lemma}\label{singular condition on F and Phi}
Let $d \geq c$. If the conditions (1), (2), (3), (5) and (6) hold of Proposition \ref{condition of nonsingularity}, then a general member of $(\mathfrak{d}_{G})_{d}$ can be singular only at the any point in $V(F_{G}) \cap V(\Phi_{G})$.
\end{lemma}

Fix $G = \mathcal{V}$, $\mathcal{I}$ or $\mathcal{K}$. From the explicit values of $a$, $b$ and $c$, we obtain a necessary and sufficient condition on $d \geq c$ for which there exists a nonsingular element of $(\mathfrak{d}_{G})_{d}$ invariant under $G$ by Proposition \ref{condition of nonsingularity}. In the following subsections, we look at the result for each group $G$ in detail.

\subsection{The Valentiner group $\mathcal{V}$}

In this subsection, we find degrees $d$ of nonsingular curves whose automorphism groups are the Valentiner group $\mathcal{V}$. We note that $d$ must be a multiple of $6$ for a $\mathcal{V}$-invariant irreducible curve of degree $d$ to exist. We first consider the case $d \leq 30$.

\begin{note}
Let $f_{1}, \cdots, f_{r}$ be homogeneous polynomials of the same degree. We write the linear system generated by $f_{1}, \cdots, f_{r}$ as $\left< f_{1}, \cdots, f_{r} \right>$.
\end{note}

\begin{lemma}\label{lowdeg nonsingularity of V}
If $d$ is $18$ or $24$, then any element of $(\mathfrak{d}_{\mathcal{V}})_{d}$ is reducible or nonreduced.
\end{lemma}

\begin{proof}
We have
\[
(\mathfrak{d}_{\mathcal{V}})_{18} = \left< F_{\mathcal{V}}^{3}, F_{\mathcal{V}}\Phi_{\mathcal{V}} \right>
\]
and
\[
(\mathfrak{d}_{\mathcal{V}})_{24} = \left< \Phi_{\mathcal{V}}^{2}, F_{\mathcal{V}}^{2}\Phi_{\mathcal{V}}, F_{\mathcal{V}}^{4} \right>.
\]
We see that any element of $(\mathfrak{d}_{\mathcal{V}})_{18}$ is divisible by $F_{\mathcal{V}}$. On the other hand, a general element of $(\mathfrak{d}_{\mathcal{V}})_{24}$ is defined by a homogeneous polynomial which can be factorized as
\[
\alpha (\Phi - \sigma_{1}F_{\mathcal{V}}^{2})(\Phi - \sigma_{2}F_{\mathcal{V}}^{2}).
\]
Therefore, any element of $(\mathfrak{d}_{\mathcal{V}})_{24}$ is reducible.
\end{proof}

\begin{theorem}\label{nonsingularity of V}
There exists a nonsingular projective plane curve of degree $d$ invariant under the Valentiner group $\mathcal{V}$ if and only if $d \equiv 0, 6$ or $12 \mod 30$.

Furthermore, if $C$ is a nonsingular projective plane curve whose automorphism group contains the Valentiner group $\mathcal{V}$, then the automorphism group of $C$ is $\mathcal{V}$.
\end{theorem}

\begin{proof}
For $d = 6$, $12$ and $30$, we have the nonsingular curves $V(F_{\mathcal{V}})$, $V(\Phi_{\mathcal{V}})$ and $V(\Psi_{\mathcal{V}})$ by Lemma \ref{basic prop of invariants} (5). By Lemma \ref{lowdeg nonsingularity of V}, the theorem is satisfied for $d \leq 30$.

Let $d > 30$. By Proposition \ref{condition of nonsingularity}, there is a nonsingular element of $(\mathfrak{d}_{\mathcal{V}})_{d}$ if and only if all of the following hold:
\begin{enumerate}
\item[(1)] There is a pair of nonnegative integers  $(j, k)$ such that $d = 12j + 30k$.
\item[(2)] There is a pair of nonnegative integers  $(i, k)$ such that $d = 6i + 30k$.
\item[(3)] There is a pair of nonnegative integers  $(i, j)$ such that $d = 6i + 12j$.
\item[(4)] $d \equiv 0, 6$ or $12 \mod 30$.
\item[(5)] $d \equiv 0, 6$ or $30 \mod 12$.
\item[(6)] $d \equiv 0, 12$ or $30 \mod 6$.
\end{enumerate}
Then the condition (4) is exactly the arithmetic condition we are considering, and so we have $d \equiv 0, 6$ or $12 \mod 30$ if there is a nonsingular element of $(\mathfrak{d}_{\mathcal{V}})_{d}$.

Assume the condition (4). Then we show that the conditions (1), (2), (3), (5) and (6) hold and hence that there is a nonsingular element of $(\mathfrak{d}_{\mathcal{V}})_{d}$. Each of the conditions (2), (3), (5) and (6) is equivalent to the condition that $d$ is divisible by $6$, and so is satisfied. Thus it suffices to show that the condition (4) implies (1).

If $d \equiv 0 \mod 30$, then there is a positive integer $k_{0}$ such that $d = 30 k_{0}$. If $d \equiv 12 \mod 30$, then there is a positive integer $k_{1}$ such that $d = 12 \cdot 1 + 30 k_{1}$. If $d \equiv 6 \mod 30$, then $k_{2} := \displaystyle{\frac{d - 36}{30}}$ is a nonnegative integer and $d$ satisfies $d = 12\cdot 3 + 30 k_{2}$. Therefore, the condition (1) also holds.

Finally, we assume that $C$ be a nonsingular projective plane curve of degree $d$ invariant under $\mathcal{V}$. Then $\mathcal{V}$ is a subgroup of the automorphism group $\Aut C$. By Remark \ref{rem of primitive}, $\Aut C$ is equal to $\mathcal{V}$.
\end{proof}

\subsection{The icosahedral group $\mathcal{I}$}

In this subsection, we find the necessary and sufficient condition on $d$ for the existence of a nonsingular curve of degree $d$ invariant under $\mathcal{I}$. By Proposition \ref{generators of irreducible invariants} and Table \ref{tab:degree}, we note that $d$ must be even for an $\mathcal{I}$-invariant irreducible curve of degree $d$ to exist. First, we look at the case of low degrees. (For $d = 14$, we will use in Section 4.)

\begin{lemma}\label{reducible of I}
\begin{enumerate}
\item[(1)] Any element of $(\mathfrak{d}_{\mathcal{I}})_{4}$ is nonreduced.
\item[(2)] If $d = 8$ or $14$, then any element of $(\mathfrak{d}_{\mathcal{I}})_{d}$ is reducible or nonreduced.
\end{enumerate}
\end{lemma}

\begin{proof}
We have
\[
(\mathfrak{d}_{\mathcal{I}})_{4} = \left< F_{\mathcal{I}}^{2} \right>,
\]
\[
(\mathfrak{d}_{\mathcal{I}})_{8} = \left< F_{\mathcal{I}}\Phi_{\mathcal{I}}, F_{\mathcal{I}}^{4} \right>
\]
and
\[
(\mathfrak{d}_{\mathcal{I}})_{14} = \left< F_{\mathcal{I}}^{2}\Psi_{\mathcal{I}}, F_{\mathcal{I}}\Phi_{\mathcal{I}}^{2}, F_{\mathcal{I}}^{4}\Phi_{\mathcal{I}}, F_{\mathcal{I}}^{7} \right>.
\]
For $d = 8$ or $14$, we clearly see that any element of $(\mathfrak{d}_{\mathcal{I}})_{d}$ is divisible by $F_{\mathcal{I}}$. Thus the claims holds.
\end{proof}

Using Proposition \ref{condition of nonsingularity}, we show the following theorem:

\begin{theorem}\label{nonsingularity of I}
There exists a nonsingular projective plane curve of degree $d$ invariant under the icosahedral group $\mathcal{I}$ if and only if $d \equiv 0, 2$ or $6 \mod 10$.
\end{theorem}

\begin{proof}
For $d = 2$, $6$ and $10$, we have the nonsingular curves $V(F_{\mathcal{I}})$, $V(\Phi_{\mathcal{I}})$ and $V(\Psi_{\mathcal{I}})$ by Lemma \ref{basic prop of invariants} (5). By Lemma \ref{reducible of I}, this theorem holds for $d \leq 10$.

Assume that $d \geq 10$. By Proposition \ref{condition of nonsingularity}, there is a nonsingular element of $(\mathfrak{d}_{\mathcal{I}})_{d}$ if and only if all of the following hold:
\begin{enumerate}
\item[(1)] There is a pair of nonnegative integers  $(j, k)$ such that $d = 6j + 10k$.
\item[(2)] There is a pair of nonnegative integers  $(i, k)$ such that $d = 2i + 10k$.
\item[(3)] There is a pair of nonnegative integers  $(i, j)$ such that $d = 2i + 6j$.
\item[(4)] $d \equiv 0, 2, 6 \mod 10$.
\item[(5)] $d \equiv 0, 2, 10 \mod 6$.
\item[(6)] $d \equiv 0, 6, 10 \mod 2$.
\end{enumerate}
Then the condition (4) is exactly the arithmetic condition we are considering. Hence, if there is a nonsingular element of $(\mathfrak{d}_{\mathcal{I}})_{d}$, then $d$ satisfies that $d \equiv 0$, $2$ or $6\mod 10$.

Assume the condition (4). Then we show that the other conditions (1), (2), (3), (5) and (6) holds. If $d$ is even, then $d$ clearly satisfies the conditions (2), (3), (5) and (6). We need to show the condition (1).

If $d \equiv 0 \mod 10$, then there is a positive integer $k_{0}$ such that $d = 10 k_{0}$. If $d \equiv 6 \mod 10$, then there is a nonnegative integer $k_{1}$ such that $d = 6 + 10 k_{1}$. If $d \equiv 2 \mod 10$, then $k_{2} := \displaystyle{\frac{d - 12}{10}}$ is a nonnegative integer and $d$ satisfies that $d \equiv 6 \cdot 2 + 10 k_{2}$. Therefore, the condition (1) holds.
\end{proof}

\begin{remark}
By Remark \ref{rem of primitive}, the icosahedral group $\mathcal{I}$ is not a maximal finite primitive subgroup of $\PGL(3, \mathbb{C})$. Therefore, the automorphism group $\Aut C$ of a nonsingular curve invariant under $\mathcal{I}$ can be bigger than $\mathcal{I}$. In particular, since $\mathcal{K}$, $H_{216}$ and its subgroup of order $36$ and $72$ do not contain a group which is conjugate to $\mathcal{I}$, $\Aut C$ is $\mathcal{I}$ or is conjugate to $\mathcal{V}$.
\end{remark}

\subsection{The Klein group $\mathcal{K}$}

In this subsection, we consider the case of the Klein group $G = \mathcal{K}$. By Proposition \ref{generators of irreducible invariants} and Table \ref{tab:degree}, we note that $d$ must be even for a $\mathcal{K}$-invariant irreducible curve of degree $d$ to exist. First, we obtain the following lemma for low degrees. (For $d = 16$ and $22$, we will use in Section 4.)

\begin{lemma}\label{lowdeg of K}
\begin{enumerate}
\item[(1)] $(\mathfrak{d}_{\mathcal{K}})_{2}$ is empty.
\item[(2)] Any element of $(\mathfrak{d}_{\mathcal{K}})_{8}$ is nonreduced.
\item[(3)] If $d$ is $10$, $16$ or $22$, then any element of $(\mathfrak{d}_{\mathcal{K}})_{d}$ is reducible or nonreduced.
\item[(4)] Any element of $(\mathfrak{d}_{\mathcal{K}})_{12}$ is singular.
\end{enumerate}
\end{lemma}

\begin{proof}
First, since the minimal degree of $\mathcal{K}$-invariant homogeneous polynomials is $4$, the claim (1) clearly holds.

We prove the claim (2). Suppose that $d = 8$. The linear system $(\mathfrak{d}_{\mathcal{K}})_{8}$ is generated only by $F_{\mathcal{K}}^{2}$. Thus any element of $(\mathfrak{d}_{\mathcal{K}})_{8}$ is nonreduced.

Next, let $d$ be $10$, $16$ or $22$. Then there is no pair of nonnegative integers $(j, k)$ such that $d = 6j + 14k$. By Lemma \ref{reducible condition}, any element of $(\mathfrak{d}_{\mathcal{K}})_{d}$ is divisible by $F_{\mathcal{K}}$. In particular, it is reducible or nonreduced, and the claim (3) holds.

Finally, we show the claim (4). We can see
\[
(\mathfrak{d}_{\mathcal{K}})_{12} = \left< \Phi_{\mathcal{K}}^{2}, F_{\mathcal{K}}^{3} \right>.
\] 
Thus its element is singular on $V(F_{\mathcal{K}}) \cap V(\Phi_{\mathcal{K}})$.
\end{proof}

From Lemma \ref{lowdeg of K} and Proposition \ref{condition of nonsingularity}, we show the following theorem for the Klein group.

\begin{theorem}\label{nonsingularity of K}
Let $d$ be a positive integer. There exists a nonsingular projective plane curve of degree $d$ invariant under the Klein group $\mathcal{K}$ if and only if $d \equiv 0, 4$ or $6 \mod 14$.

Moreover, if $C$ is a nonsingular projective plane curve whose automorphism group contains the Klein group $\mathcal{K}$, then the automorphism group of $C$ is $\mathcal{K}$.
\end{theorem}

\begin{proof}
For $d = 4$, $6$ and $14$, we have the nonsingular curves $V(F_{\mathcal{K}})$, $V(\Phi_{\mathcal{K}})$ and $V(\Psi_{\mathcal{K}})$ by Lemma \ref{basic prop of invariants} (5). By Lemma \ref{lowdeg of K}, the claim holds for $d \leq 14$.

Let $d > 14$. By Proposition \ref{condition of nonsingularity}, there exists a nonsingular curve of degree $d$ invariant under $\mathcal{K}$ if and only if $d$ satisfies the following conditions:
\begin{enumerate}
\item[(1)] There is a pair of nonnegative integers  $(j, k)$ such that $d = 6j + 14k$.
\item[(2)] There is a pair of nonnegative integers  $(i, k)$ such that $d = 4i + 14k$.
\item[(3)] There is a pair of nonnegative integers  $(i, j)$ such that $d = 4i + 6j$.
\item[(4)] $d \equiv 0, 4, 6 \mod 14$.
\item[(5)] $d \equiv 0, 4, 14 \mod 6$.
\item[(6)] $d \equiv 0, 6, 14 \mod 4$.
\end{enumerate}
Then the condition (4) is exactly the arithmetic condition we are considering. If there is a nonsingular element of $(\mathfrak{d}_{\mathcal{K}})_{d}$, then the condition (4) holds.

Assume the condition (4). We show that the other conditions (1), (2), (3), (5) and (6) hold. First, we easily see that the conditions (3), (5) and (6) hold for even $d \geq 4$.

If $d \equiv 0$ or $6 \mod 14$, then there clearly exists a pair of nonnegative integers $(j, k)$ with $d = 6j + 14k$. If $d \equiv 4 \mod 14$, then $k_{1} := \displaystyle{\frac{d - 18}{14}}$ is a nonnegative integer and $d$ satisfies $d = 3 \cdot 6 + 14 k_{1}$. Thus the condition (1) holds.

If $d \equiv 0$ or $4 \mod 14$, then there clearly exists a pair of nonnegative integers $(i, k)$ such that $d = 4i + 14k$. If $d \equiv 6 \mod 14$, then $k_{2} := \displaystyle{\frac{d - 20}{14}}$ is a nonnegative integer and $d$ satisfies $d = 5 \cdot 4 + 14k_{2}$. Hence, the condition (2) holds.

Finally, we assume that $C$ is a nonsingular projective plane curve invariant under the Klein group $\mathcal{K}$. Then $\mathcal{K}$ is a subgroup of $\Aut C$. By Remark \ref{rem of primitive}, we have $\Aut C = \mathcal{K}$.
\end{proof}

\section{Integral curves invariant under a simple primitive group}

In this section, we find integral (i.e. irreducible and reduced) curves invariant under $G = \mathcal{V}, \mathcal{I}$ or $\mathcal{K}$. By Theorem \ref{nonsingularity of V}, Theorem \ref{nonsingularity of I} and Theorem \ref{nonsingularity of K}, we know in which degree there is a nonsingular curve invariant under $G$. Such a curve is in particular integral. Thus we have only to consider the case where $(\mathfrak{d}_{G})_{d}$ is nonempty and any member of $(\mathfrak{d}_{G})_{d}$ is singular: Such a degree $d$ satisfies
\[
d \equiv
\left\{
\begin{array}{ll}
18 \mbox{{\ }or } 24 \mod 30 & \mbox{for } G = \mathcal{V}, \\
4 \mbox{{\ }or } 8 \mod 10 & \mbox{for } G = \mathcal{I}, \\
2, 8, 10 \mbox{{\ }or } 12 \mod 14 \mbox{ } (d \geq  8) & \mbox{for } G = \mathcal{K}. \\
\end{array}
\right.
\]

First, we study the type of singularities of a general invariant curve of a given degree.

\begin{note}
When positive integers $a, b$ and $c$ are fixed (i.e. we fix the group $G$), we define the set
\[
I_{d} := \left\{(i, j) \middle| i, j \geq 0 \mbox{ and } \frac{d - (ai + bj)}{c} \mbox{ is a nonnegative integer.} \right\}
\]
for a positive integer $d$.
\end{note}

\begin{lemma}\label{lemma of sing type}
Let $P$ be a point in $V(F_{G}) \cap V(\Phi_{G})$. Take a general element of $(\mathfrak{d}_{G})_{d}$ defined by a homogeneous polynomial
\[
H(x, y, z) = \sum_{(i, j) \in I_{d}} c_{i j} F_{G}(x, y, z)^{i} \Phi_{G}(x, y, z)^{j} \Psi_{G}(x, y, z)^{\frac{d-(ai+bj)}{c}}.
\]
Then  the singularity of the curve $V(H)$ at the point $P$ is analytically equivalent to the singularity of the formal curve $V(h) \subset \mathbb{A}^{2}$ at the origin where the power series $h \in \mathbb{C}[\![s, t]\!]$ is defined by
\[
h(s, t) = \sum_{(i, j) \in I_{d}} c_{i j} \lambda_{i j}(s, t) s^{i} t^{j}
\]
where $\lambda_{i j}$ is an analytic function such that $\lambda_{i j}(0, 0) \neq 0$.
\end{lemma}

\begin{proof}
Put $x_{1} = x$, $x_{2} = y$ and $x_{3} = z$ and suppose that the point $P$ is in $\{ x_{r} \neq 0 \}$ for an index $r = 1$, $2$ or $3$. Take regular functions $s := \dfrac{F_{G}}{x_{r}^{a}}$ and $t := \dfrac{\Phi_{G}}{x_{r}^{b}}$. By Lemma \ref{basic prop of invariants} (6), since the curves $V(F_{G})$ and $V(\Phi_{G})$ meet transversally at $P$, this pair $(s, t)$ is an analytic coordinate system in an affine neighborhood of the origin in $\mathbb{A}^{2}$. Since the homogeneous polynomial $\Psi_{G}$ dose not vanish on $V(F_{G}) \cap V(\Phi_{G})$ by Lemma \ref{basic prop of invariants} (4), we can take an analytic function $\lambda_{i j}(s, t)$ such that
\[
\lambda_{i j}(s(x, y, z), t(x, y, z)) = \left( \frac{\Psi_{G}(x, y, z)}{(x_{r}(x, y, z))^{c}} \right)^{\frac{d-(ai+bj)}{c}}
\]
for any $(x, y, z)$ near $P$, and $\lambda_{i j}(0, 0) \neq 0$. Then the formal curve defined by
\[
h(s, t) := \sum_{(i, j) \in I_{d}} c_{i j} \lambda_{i j}(s, t) s^{i} t^{j}
\]
passes the origin and on a neighborhood at the origin is isomorphic to the curve $V(H)$ on a neighborhood at $P$.  In particular, since the curve $V(H)$ is a general element, the singularity type of $V(H)$ at the point $P$ is equal to the singularity type of $V(h)$ at the origin.
\end{proof}

Assume that $d$ is a positive integer such that there is no nonsingular curve invariant under $G$ and the linear system$(\mathfrak{d}_{G})_{d}$ is nonempty. Let $H$ be a general homogeneous polynomial of degree $d$ invariant under $G$.

\begin{lemma}\label{singularities for invariants}
For each group and any degree $d$ in the following list, all singularities of the curve $V(H)$ belong to $V(F_{G}) \cap V(\Phi_{G})$, and are of the same type given in the list.
\[
\begin{array}{|c|ccc|}
\hline
G & \multicolumn{2}{c}{\mbox{the condition on }d} & \mbox{type of singularities} \\
\hline
\mathcal{V} & d \geq 48 & d \equiv 18 \mod 30 & A_{1} \mbox{ (node)} \\
& d \geq 54 &  d \equiv 24 \mod 30 & A_{3} \mbox{ (tacnode)} \\
\hline
\mathcal{I} & d \geq 24 & d \equiv 4 \mod 10 & A_{3} \mbox{ (tacnode)} \\
& d \geq 18 & d \equiv 8 \mod 10 & A_{1} \mbox{ (node)} \\
\hline
\mathcal{K} & d \geq 30 & d \equiv 2 \mod 14 & D_{4} \\
& d \geq 36 & d \equiv 8 \mod 14 & A_{5} \\
& d \geq 24 & d \equiv 10 \mod 14 & A_{1} \mbox{ (node)} \\
& d \geq 12 & d \equiv 12 \mod 14 & A_{2} \mbox{ (cusp)} \\
\hline
\end{array}
\]
\end{lemma}

\begin{proof}
First, a general element of $(\mathfrak{d}_{\mathcal{K}})_{12}$ is singular only on $V(F_{\mathcal{K}}) \cap V(\Phi_{\mathcal{K}})$ by the proof of Lemma \ref{lowdeg of K} (4). For the other cases in the list, we have $d \geq c$. Let us prove that the curve $V(H)$ has singular points only on $V(F_{G}) \cap V(\Phi_{G})$ in these cases. We recall the conditions (1) to (6) in the proof of Theorem \ref{nonsingularity of V}, Theorem \ref{nonsingularity of I} and Theorem \ref{nonsingularity of K}. We know that the condition (4) does not hold. By Lemma \ref{singular condition on F and Phi}, it suffices to show the conditions (1), (2), (3), (5) and (6) hold for each group $G = \mathcal{V}$, $\mathcal{I}$ or $\mathcal{K}$.

Let $G = \mathcal{V}$. Since $d$ is a multiple of 6, the conditions (2), (3), (5) and (6) (cf. the proof of Theorem \ref{nonsingularity of V}) are satisfied. We prove the condition (1), i.e., there is a pair of nonnegative integers $(j, k)$ such that $d = 12j + 30k$. If $d \equiv 18 \mod 30$, then the pair of integers $(j, k) = \left(4, \dfrac{d-48}{30} \right)$ satisfies $d = 12j + 30k$. If $d \equiv 24 \mod 30$, then $(j, k) = \left(2, \dfrac{d-24}{30} \right)$ satisfies $d = 12j + 30k$. Hence, the degree $d$ satisfies the condition (1).

Let $G = \mathcal{I}$. Since $d$ is even, the conditions (2), (3), (5) and (6) (cf. the proof of Theorem \ref{nonsingularity of I}) hold. We show the condition (1), i.e., there is a pair of integers $(j, k)$ such that $d = 6j + 10k$. If $d \geq 24$ and $d \equiv 4 \mod 10$, then $(j, k) = \left(4, \dfrac{d-24}{10}\right)$ satisfies $d = 6j + 10k$. If $d \equiv 8 \mod 10$, then $(j, k) = \left(3, \dfrac{d-18}{10}\right)$ satisfies $d = 6j + 10k$. Thus the degree $d$ satisfies the condition (1).

Let $G = \mathcal{K}$. If the integer $d$ is even and greater than $4$, then the conditions (3), (5) and (6) (cf. the proof of Theorem \ref{nonsingularity of K}) hold. We check the conditions (1) and (2). To show the condition (1), i.e., that there is a pair of nonnegative integers $(j, k)$ such that $6j + 14k = d$ for $d \equiv 2, 8, 10$ or $12 \mod 14$, we can take $(j, k) = \left(5, \dfrac{d - 30}{14}\right)$, $\left(6, \dfrac{d - 36}{14}\right)$, $\left(4, \dfrac{d - 24}{14}\right)$ or $\left(2, \dfrac{d - 12}{14}\right)$, respectively. To show the condition (2), i.e., that there is a pair of nonnegative integers $(i, k)$ such that $4i + 14k =d$ for $d \equiv 2, 8, 10$ or $12 \mod 14$, we can take $(i, k) = \left(4, \dfrac{d - 16}{14}\right)$, $\left(9, \dfrac{d - 36}{14}\right)$, $\left(6, \dfrac{d - 24}{14}\right)$ or $\left(3, \dfrac{d - 12}{14}\right)$, respectively.

Next, we look at singularities of the curve $V(H)$. Let $P$ be a point in $V(F_{G}) \cap V(\Phi_{G})$. By Lemma \ref{lemma of sing type}, the singular type of $V(H)$ at $P$ is analytically equivalent to the singularity of the formal curve $V(h)$ at the origin where
\[
h(s, t) = \sum_{(i, j) \in I_{d}} c_{i j}\lambda_{i j}(s, t) s^{i}t^{j}
\]
such that $c_{i j}$ is general and $\lambda_{i j}(0, 0) \neq 0$.

We consider the case of nodes. If $(1, 1) \in I_{d}$, then a general element $V(h)$ has a node at the origin. Since
\[
\frac{d-(a+b)}{c} = \left\{\begin{array}{ccccc}
\dfrac{d - 18}{30} & (d \geq 48, & d \equiv 18 \mod 30 & \mbox{and} & G = \mathcal{V}) \\
&&&&\\
\dfrac{d - 8}{10} & (d \geq 18, & d \equiv 8 \mod 10 & \mbox{and} & G = \mathcal{I}) \\
&&&&\\
\dfrac{d - 10}{14} & (d \geq 24, & d \equiv 10 \mod 14 & \mbox{and} & G = \mathcal{K}) \\
\end{array}\right.
\]
is a nonnegative integer, a general element of $(\mathfrak{d}_{\mathcal{V}})_{d}$ has nodal.

We consider the case of cusps. If $(3, 0)$ and $(0, 2) \in I_{d}$ and any $(i, j) \in I_{d}$ satisfies $2i + 3j \geq 6$, then we can transform the curve $V(h)$ to the curve defined by $s^{3} + t^{2} = 0$ by a suitable change of variables. Assume that $G = \mathcal{K}$ and $d \equiv 12 \mod 14$. For $(i, j) = (3, 0)$ and $(0, 2)$, $\dfrac{d - (4i + 6j)}{14}$ is a nonnegative integer, so $(3, 0)$ and $(0, 2) \in I_{d}$. If $(i, j) \in I_{d}$, then $2i + 3j \equiv 6 \mod 7$, and we see that $2i + 3j \geq 6$.

We consider the case of tacnodes. If $(4, 0)$ and $(0, 2) \in I_{d}$ and any $(i, j) \in I_{d}$ satisfies $i + 2j \geq 4$, then we can transform the curve $V(h)$ to the curve define by $s^{4} + t^{2} = 0$ by a suitable change of variables. We can obtain this fact if we interchange $i$ and $j$. Assume that $G = \mathcal{V}$, $d \geq 54$ and $d \equiv 24 \mod 30$. For $(i, j) = (4, 0)$ or $(0, 2)$, $\dfrac{d - (6i + 12j)}{30}$ is a positive integer. Thus $(4, 0)$ and $(0, 2) \in I_{d}$. For any $(i, j) \in I_{d}$, since $i + 2j \equiv 4 \mod 5$, we see $i + 2j \geq 4$. Assume that $G = \mathcal{I}$, $d \geq 24$ and $d \equiv 4 \mod 10$. For $(i, j) = (2, 0)$ or $(0, 4)$, $\dfrac{d - (2i + 6j)}{10}$ is a nonnegative integer, i.e., $(2, 0)$ and $(0, 4) \in I_{d}$. Any $(i, j) \in I_{d}$ satisfies $i + 3j \equiv 2 \mod 5$. By multiplying both side by $2$, we have $2i + j \equiv 4 \mod 5$, and $2i + j \geq 4$.

We consider the case of type $A_{5}$. If $(2, 0)$ and $(0, 6) \in I_{d}$ and any $(i, j) \in I_{d}$ satisfies $3i + j \geq 6$, then we can transform the curve $V(h)$ to the curve define by $s^{2} + t^{6} = 0$ by a suitable change of variables. Assume that $G = \mathcal{K}$, $d \geq 36$ and $d \equiv 8 \mod 14$. For $(i, j) = (2, 0)$ and $(0, 6)$, $\dfrac{d - (4i + 6j)}{14}$ is a nonnegative integer. Any $(i, j) \in I_{d}$ satisfies $2i + 3j \equiv 4 \mod 7$. By multiplying both side by $5$, we have $3i + j \equiv 6 \mod 7$, and $3i + j \geq 6$.

We consider the case of type $D_{5}$. If $(4, 0)$ and $(1, 2)$ and any $(i, j) \in I_{d}$ satisfies $2i + 3j \geq 8$, then we can transform the curve $V(h)$ to the curve define by $s^{4} + s t^{2} = 0$ by a suitable change of variables. Assume that $G = \mathcal{K}$, $d \geq 30$ and $d \equiv 2 \mod 14$. 
For $(i, j) = (4, 0)$ and $(1, 2)$, $\dfrac{d - (4i + 6j)}{14}$ is a positive integer. Any $(i, j) \in I_{d}$ satisfies $2i + 3j \equiv 1 \mod 7$. We see $2i + 3j \geq 8$.
\end{proof}

If any singularity of the curve $V(H)$ is a cusp, then $V(H)$ is irreducible. We have the following if a general invariant curve $V(H)$ is reducible.

\begin{lemma}\label{irreducible components are nonsingular}
For $d$ as in Lemma \ref{singularities for invariants}, if the curve $V(H)$ is reducible, then the following hold.
\begin{enumerate}
\item[(1)] We can write $H = H_{1} \cdots H_{n} \in \mathbb{C}[x, y, z]$ where $V(H_{1})$, $V(H_{2})$, $\cdots$, $V(H_{n})$ are pairwise distinct integral curves for $n \geq 2$. 
\item[(2)] As a set, we have
\[
\Sing V(H) = V(F_{G}) \cap V(\Phi_{G}) = \coprod_{i \neq j} (V(H_{i}) \cap V(H_{j})).
\]
Furthermore, the intersection multiplicity is the same integer $m$ at any point in this set where
\[
m = \left\{\begin{array}{cl}
1 & \mbox{if the singularity is a node},\\
2 & \mbox{if the singularity is a tacnode or of type } D_{5},\\
3 & \mbox{if the singularity is of type } A_{5}.
\end{array}\right.
\]
\item[(3)] Let $d_{i} = \deg H_{i}$. Then we obtain the formula
\[
\sum_{i < j} d_{i} d_{j} = m \deg F_{G} \deg \Phi_{G}.
\]
In particular, for any $k$ and $l$ with $k \neq l$, we have the inequalities
\begin{equation}\label{first inequality}
d_{k} (d - d_{k}) \leq m \deg F_{G} \deg \Phi_{G} \tag{$\ast$}
\end{equation}
with equality only for $n = 2$ and
\begin{equation}\label{second inequality}
d_{l} ((d - d_{k}) - d_{l}) \leq m \deg F_{G} \deg \Phi_{G} - d_{k} (d - d_{k}). \tag{$\ast\ast$}
\end{equation}
\end{enumerate}
\end{lemma}

\begin{proof} 
First, since the singular locus $V(F_{G}) \cap V(\Phi_{G})$ of $(\mathfrak{d}_{G})_{d}$ is finite, $V(H)$ is reduced. Thus we can write $H = H_{1} \cdots H_{n} \in \mathbb{C}[x, y, z]$ as in (1).

Next, we prove the claim (2). By Lemma \ref{singularities for invariants}, we have $\Sing V(H) = V(F_{G}) \cap V(\Phi_{G})$. Any intersection point $P \in V(H_{i}) \cap V(H_{j})$ is a singular point of $V(H)$, i.e., $\bigcup\limits_{i \neq j} (V(H_{i}) \cap V(H_{j})) \subset \Sing V(H)$. Conversely, take any $P \in \Sing V(H)$ and a point $Q \in V(H_{1}) \cap V(H_{2})$ (since $n \geq 2$). Then $P$ and $Q$ are contained in $V(F_{G}) \cap V(\Phi_{G})$. By Lemma \ref{basic prop of invariants} (7), since $V(F_{G}) \cap V(\Phi_{G})$ is a $G$-orbit, there is an action $A \in G$ which send $Q$ to $P$. In addition, $A$ induces the permutation of the curves $V(H_{1})$, $V(H_{2})$, $\cdots$, $V(H_{n})$. Thus $P$ is an intersection point of $V(H_{i})$ and $V(H_{j})$ where $A$ send $V(H_{1})$ to $V(H_{i})$ and $V(H_{2})$ to $V(H_{j})$. Finally, since any singularity of $V(H)$ has a $2$ analytic branch by Lemma \ref{singularities for invariants}, only two curves intersect at any intersection point. Therefore, $\bigcup\limits_{i \neq j} (V(H_{i}) \cap V(H_{j}))$ is disjoint. In particular, since the singularities are a same type, the intersection multiplicity is a same integer $m$ which holds as in (2) by Lemma \ref{singularities for invariants}.

Finally, we show the claim (3). By B\'{e}zout's theorem, the degrees $d_{1}$, $d_{2}$, $\cdots$, $d_{n}$ satisfy the first formula in (3). (Note that $\#(V(F_{G}) \cap V(\Phi_{G})) = \deg F_{G} \deg \Phi_{G}$ by Lemma \ref{basic prop of invariants} (6) and (7).) Fix $d_{k}$ and $d_{l}$ with $i \neq j$. We see the inequality
\[
\begin{array}{ccccc}
d_{k} (d - d_{k}) & = & d_{k} \displaystyle\sum_{i \neq k} d_{i} & & \\
& = & \displaystyle\sum_{i \neq k} d_{i} d_{k} & \leq & \displaystyle\sum_{i < j} d_{i} d_{j}.
\end{array}
\]
Hence, the inequality (\ref{first inequality}) holds. Since $d_{i} d_{j} > 0$ for any $i$ and $j$ with $i \neq j$, the two sides are equal if and only if $n = 2$. On the other hand, we have
\[
\sum_{i < j, i \neq k, j \neq k} d_{i} d_{j} = m \deg F_{G} \deg \Phi_{G} - \sum_{i \neq k} d_{i} d_{k}
\]
from the first equality in (3) and the inequality
\[
\begin{array}{ccccc}
d_{l} ((d - d_{k}) - d_{l}) & = & d_{l} \displaystyle\sum_{i \neq k, i \neq l} d_{i} & & \\
& = & \displaystyle\sum_{i \neq k, i \neq l} d_{i}d_{l} & \leq & \displaystyle\sum_{i < j, i \neq k, j \neq k} d_{i} d_{j}.
\end{array}
\]
Therefore, the inequality (\ref{second inequality}) also holds.
\end{proof}

In the following subsections, we prove the irreducibility of curves for each group by a contradiction: If there are at least $2$ irreducible components, then the intersection numbers of the components imply that the curve has too much singularity.

\subsection{The Valentiner group $\mathcal{V}$}

In this subsection, we deal with the case of the Valentiner group. We show the following theorem related to integral curves under the Valentiner group $\mathcal{V}$.

\begin{theorem}\label{irreducible of V}
There exists an integral projective plane curve of degree $d$ invariant under the Valentiner group $\mathcal{V}$ if and only if $d$ is a multiple of $6$, $d \neq 18$ and $d \neq 24$.
\end{theorem}

\begin{proof}
There is a possibly reducible or nonreduced curve of degree $d$ invariant under the Valentiner group $\mathcal{V}$ if and only if $d$ is divisible by $6$. By Theorem \ref{nonsingularity of V}, if $d \equiv 0, 6, 12 \mod 30$, then the linear system $(\mathfrak{d}_{\mathcal{V}})_{d}$ has a nonsingular element. On the other hand, by Lemma \ref{lowdeg nonsingularity of V}, the linear system $(\mathfrak{d}_{\mathcal{V}})_{d}$ has only reducible or nonreduced curves for $d = 18$ or $24$. We have to show that there is an irreducible element of $(\mathfrak{d}_{\mathcal{V}})_{d}$ if the degree $d$ satisfies $d > 30$ and $d \equiv 18$ or $24 \mod 30$.

Take a general member $H$ of the linear system $(\mathfrak{d}_{\mathcal{V}})_{d}$. We assume that the homogeneous polynomial $H$ is reducible and derive a contradiction. By Lemma \ref{irreducible components are nonsingular} (1), we can write $H = H_{1}\cdots H_{n}$ where the curves $V(H_{i})$ are pairwise distinct integral curves with $n \geq 2$.

Let $d_{i} = \deg H_{i}$. We may assume that $d_{1} \leq d_{2} \leq \cdots \leq d_{n}$. By Lemma \ref{irreducible components are nonsingular} (3) (\ref{first inequality}), for the degree $d_{1}$, we obtain
\begin{equation}\label{iq for V}
d_{1}(d - d_{1}) \leq 72m \tag{$\ast$}
\end{equation}
where $m$ is an integer as in Lemma \ref{irreducible components are nonsingular} (2). We note that the left hand side is monotonically increasing for $1 \leq d_{1} \leq \dfrac{d}{2}$, in particular, and it is greater than or equal to $d - 1$.

Suppose that $d \equiv 18 \mod 30$. Then $m = 1$ by Lemma \ref{singularities for invariants}. If $d \geq 18 + 2 \cdot 30 = 78$, then the inequality (\ref{iq for V}) does not hold. Thus $d$ is only $48$. For $2 \leq d_{1} \leq 24$, the inequality (\ref{iq for V}) is not satisfied. Let $d_{1} = 1$. If $n = 2$ (i.e. $d_{2} = 47$), then both sides of (\ref{iq for V}) are equal; however, $1 \cdot 47 \neq 72 \cdot 1$. If $n > 2$, by Lemma \ref{irreducible components are nonsingular} (3) (\ref{second inequality}), we have
\[
46 \leq d_{2}(47 - d_{2}) \leq (1 \cdot 72) - (1 \cdot 47) = 25
\]
since $d_{2} \neq 47$. This is a contradiction.

Suppose that $d \equiv 24 \mod 30$. Then $m = 2$ by Lemma \ref{singularities for invariants}. Thus any intersections of the curve $V(H_{1})$ and the curve $V(H_{2} \cdots H_{n})$ is a tacnode, and we see 
\[
d_{1} (d - d_{1}) = 2 \cdot \#\left(V(H_{1}) \cap V(H_{2} \cdots H_{n})\right)
\]
by B\'{e}zout's theorem. Hence, we note that the degree $d_{1}$ is even if $d$ is even. If $d \geq 24 + 5 \cdot 30 = 174$, then the inequality (\ref{iq for V}) does not hold. Hence, $d$ is $54$, $84$, $114$ or $144$. Let $d = 84, 114$ or $144$. Since $d_{1}$ is even, $2 \leq d_{1} \leq \dfrac{d}{2}$ and the inequality (\ref{iq for V}) does not hold. Let $d = 54$. Since the inequality (\ref{iq for V}) is not satisfied for $3 \leq d_{1} \leq 27$ and the degree $d_{1}$ is even, we see that $d_{1} = 2$. If $n = 2$, then both sides of (\ref{iq for V}) are equal; however, $2 \cdot 52 \neq 72 \cdot 2$. If $n > 2$, by Lemma \ref{irreducible components are nonsingular} (3) (\ref{second inequality}), we have
\[
51 \leq d_{2}(52 - d_{2}) \leq (2 \cdot 72) - (2 \cdot 52) = 40
\]
since $d_{2} \neq 52$. This is also a contradiction.
\end{proof}

\subsection{The icosahedral group $\mathcal{I}$}

In this subsection, we see the case of the icosahedral group in the same way as Subsection 4.1.

\begin{theorem}
Let $d$ be a positive integer. There exists an integral projective plane curve of degree $d$ invariant under $\mathcal{I}$ if and only if $d$ is even and neither $4$, $8$ nor $14$.
\end{theorem}

\begin{proof}
There is a possibly reducible or nonreduced curve of degree $d$ invariant under $\mathcal{I}$ if and only if the degree $d$ is even. By Theorem \ref{nonsingularity of I}, the linear system $(\mathfrak{d}_{\mathcal{I}})_{d}$ has a nonsingular element if $d \equiv 0, 2, 6 \mod 10$. On the other hand, for $d = 4, 8$ or $14$, the linear system $(\mathfrak{d}_{\mathcal{I}})_{d}$ has no integral member by Lemma \ref{reducible of I}.

We assume that $d$ satisfies $d \geq 18$ and $d \equiv 4$ or $8 \mod 10$. Take a general element $H$ of the linear system $(\mathfrak{d}_{\mathcal{I}})_{d}$ (where the curve $V(H)$ is singular only at points in $V(F_{\mathcal{I}}) \cap V(\Phi_{\mathcal{I}})$). Assume that the homogeneous polynomial $H$ is reducible, and derive a contradiction. By Lemma \ref{irreducible components are nonsingular} (1), we can write $H = H_{1}\cdots H_{n}$ where the curves $V(H_{i})$ are pairwise distinct integral curves for $n \geq 2$.

Let $d_{i} = \deg H_{i}$. We may assume that $d_{1} \leq d_{2} \leq \cdots \leq d_{n}$. By Lemma \ref{irreducible components are nonsingular} (3) (\ref{first inequality}), for the degree $d_{1}$, we obtain
\begin{equation}\label{iq for I}
d_{1}(d - d_{1}) \leq 12m \tag{$\ast$}
\end{equation}
where $m$ is an integer as in Lemma \ref{irreducible components are nonsingular} (2). We note that the left hand side is monotonically increasing for $1 \leq d_{1} \leq \dfrac{d}{2}$, in particular, and it is greater than or equal to $d - 1$.

Suppose that $d \equiv 8 \mod 10$. Then $m = 1$ by Lemma \ref{singularities for invariants}. Then the inequality (\ref{iq for I}) does not hold for $d \geq 18$.

Suppose that $d \equiv 4 \mod 10$. Then $m = 2$ by Lemma \ref{singularities for invariants}. If $d \geq 4 + 3 \cdot 10 = 34$, then the inequality (\ref{iq for I}) does not hold. Take $d = 24$. Since the inequality (\ref{iq for I}) is not satisfied for $2 \leq d_{1} \leq 12$. Let $d_{1} = 1$. If $n = 2$ (i.e. $d_{2} = 23$), both sides of (\ref{iq for I}) are equal; however, $1 \cdot 23 \neq 12 \cdot 2$. If $n > 2$, by Lemma \ref{irreducible components are nonsingular} (3) (\ref{second inequality}), we have
\[
22 \leq d_{2}(23 - d_{2}) \leq (2 \cdot 12) - (1 \cdot 23) = 1.
\]
since $d_{2} \neq 23$. This is a contradiction.
\end{proof}

\subsection{The Klein group $\mathcal{K}$}

In this subsection, we consider irreducible curves invariant under the Klein group $\mathcal{K}$.

\begin{theorem}
Let $d$ be a positive integer. There exists an integral projective plane curve of degree $d$ invariant under $\mathcal{K}$ if and only if $d$ is even and neither $2$, $8$, $10$, $16$ nor $22$.
\end{theorem}

\begin{proof}
There is a possibly reducible or nonreduced element of the linear system $(\mathfrak{d}_{\mathcal{K}})_{d}$ if and only if the degree $d$ is even and $d \geq 4$. By Theorem \ref{nonsingularity of K}, if $d \equiv 0, 4$ or $6 \mod 14$, then there is nonsingular curves of degree $d$ invariant under the Klein group $\mathcal{K}$. On the other hand, there is no integral elements of the linear system $(\mathfrak{d}_{\mathcal{K}})_{d}$ for $d = 8, 10, 16$ and $22$ by Lemma \ref{lowdeg of K}.

We show that there is an integral element of the linear system $(\mathfrak{d}_{\mathcal{K}})_{d}$ for given degree $d$. Take a general element $H$ of the linear system $(\mathfrak{d}_{\mathcal{K}})_{d}$ (where the curve $V(H)$ has the singularities only at the points in $V(F_{\mathcal{K}}) \cap V(\Phi_{\mathcal{K}})$).

First, if $d \equiv 12 \mod 14$, then any singularity of the curve $V(H)$ is a cusp by Lemma \ref{singularities for invariants}. In this case, $V(H)$ is integral.

We consider the case of other degrees. Assume that $V(H)$ is reducible, and derive a contradiction. By Lemma \ref{irreducible components are nonsingular} (1), we can write $H = H_{1}\cdots H_{n}$ where the curves $V(H_{i})$ are pairwise distinct integral curves for $n \geq 2$. Let $d_{i} = \deg H_{i}$. We may assume that $d_{1} \leq d_{2} \leq \cdots \leq d_{n}$. By Lemma \ref{irreducible components are nonsingular} (3), for the degree $d_{1}$, we obtain
\begin{equation}\label{iq of K}
d_{1}(d - d_{1}) \leq 24m \tag{$\ast$}
\end{equation}
where $m$ is an integer as in Lemma \ref{irreducible components are nonsingular} (2). We note that the left hand side is monotonically increasing for $1 \leq d_{1} \leq \dfrac{d}{2}$, in particular, and it is greater than or equal to $d - 1$.

Suppose that $d \equiv 2 \mod 14$. Then $m = 2$ by Lemma \ref{singularities for invariants}. For $d \geq 2 + 4 \cdot 14 = 58$, the inequality (\ref{iq of K}) does not hold. Thus $d = 30$ or $d = 44$. For $2 \leq d_{1} \leq \dfrac{d}{2}$, the inequality (\ref{iq of K}) is not satisfied. Let $d_{1} = 1$. Since all intersections of the line $V(H_{1})$ and the curve $V(H_{2} \cdots V_{n})$ is of type $D_{5}$, by B\'{e}zout's theorem, we see
\[
1 \cdot (d - 1) = 2 \cdot \# \left( V(H_{1}) \cap V(H_{2} \cdots H_{n}) \right).
\]
The left hand side is odd, while the right hand side is even. This is a contradiction.

Suppose that $d \equiv 8 \mod 14$. Then $m = 3$ by Lemma \ref{singularities for invariants}. For $d \geq 8 + 5 \cdot 14 = 78$, the inequality (\ref{iq of K}) does not hold. Take  $d = 50$ or $d = 64$. For $2 \leq d_{1} \leq d - 2$, the inequality (\ref{iq of K}) is not satisfied. Let $d_{1} = 1$. If $n = 2$ (i.e. $d_{2} = d - 1$), both sides of (\ref{iq of K}) are equal; however, $1(d - 1) \neq 24 \cdot 3$. If $n > 2$, by Lemma \ref{irreducible components are nonsingular} (3) (\ref{second inequality}), we have
\[
d - 2 \leq d_{2}((d - 1) - d_{2}) \leq (3 \cdot 24) - (d - 1) = 73 - d
\]
since $d_{2} \neq d - 1$. For $d = 50$ and $64$, this inequality does not hold. Take $d = 36$. For $3 \leq d_{1} \leq 18$, the inequality (\ref{iq of K}) is not satisfied. Let $d_{1} = 2$. If $n = 2$, then both sides of (\ref{iq of K}) are equal; however, $2 \cdot 34 \neq 24 \cdot 3$. If $n > 2$, by Lemma \ref{irreducible components are nonsingular} (3) (\ref{second inequality}), we have
\[
33 \leq d_{2} (34 - d_{2}) \leq  (3 \cdot 24) -  (2 \cdot 34) = 4
\]
since $d_{2} \neq 34$. This is a contradiction. Let $d_{1} = 1$. If $n = 2$, then both sides of (\ref{iq of K}) are equal; however, $1 \cdot 35 \neq 3 \cdot 24$. If $n > 2$, in the same way, we have
\[
d_{2} (35 - d_{2}) \leq  (3 \cdot 24) - (1 \cdot 35) = 37.
\]
The left hand side is monotonically increasing for $1 \leq d_{2} \leq \dfrac{35}{2}$ and the inequality is not satisfied for $2 \leq d_{2} \leq \dfrac{35}{2}$. Let $d_{2} = 1$. Then any intersection of the two lines $V(H_{1})$ and $V(H_{2})$ is a node, while any singularity is of type $A_{5}$. This is a contradiction.

Suppose that $d \equiv 10 \mod 14$. Then $m = 1$ by Lemma \ref{singularities for invariants}. For $d \geq 4 + 3 \cdot 10 = 34$, the inequality (\ref{iq of K}) does not hold. Take $d = 24$. For $2 \leq d_{1} \leq 12$, the inequality (\ref{iq of K}) is not satisfied. Let $d_{1} = 1$. If $n = 2$, then both sides of (\ref{iq of K}) are equal; however, $1 \cdot 23 \neq 24 \cdot 1$. If $n > 2$, by Lemma \ref{irreducible components are nonsingular} (3) (\ref{second inequality}), we have
\[
22 \leq d_{2} (23 - d_{2}) \leq (1 \cdot 24) - (1 \cdot 23) = 1
\]
since $d_{2} \neq 23$. This is a contradiction.
\end{proof}

\appendix
\section{Appendix}

In this section, we explain the SINGULAR code used in the proof of Lemma \ref{basic prop of invariants}. We recall the claims which we need to check:
\begin{enumerate}
\item[(4)] The radical of the ideal generated by $F_{G}$, $\Phi_{G}$ and $\Psi_{G}$ is the maximal ideal $\left< x, y, z\right>$.
\item[(5)] For $f = F_{G}$, $\Phi_{G}$ or $\Psi_{G}$, the radical of the ideal generated by the derivatives of $f$ is the maximal ideal $\left< x, y, z\right>$.
\item[(6)] The ideal generated by $F_{G}$ and $\Phi_{G}$ is equal to its radical.
\end{enumerate}
In addition, for $G = \mathcal{I}$, we will look at the following claim.
\begin{enumerate}
\item[(1)] The Jacobian of $(F_{\mathcal{I}}, \Phi_{\mathcal{I}}, \Psi_{\mathcal{I}})$ is nonzero.
\end{enumerate}

\subsection{The Valentiner group}

We check the claims in almost the same way for all groups. In this subsection, we take $G = \mathcal{V}$ as an example. First, we note the following.

\begin{remark}
It is known that the curve defined by
\[
W(x, y, z) := 10 x^{3} y^{3} + 9 x^{5} z + 9 y^{5} z - 45 x^{2} y^{2} z^{2} -135 xyz^{4} + 27 z^{6}
\]
is invariant under the alternative group $\mathfrak{A}_{6}$ in a suitable coordinate system (\cite{Wiman}). This curve is called the \textit{Wiman curve}.

We define the matrix
\[
R :=
\left(
\begin{array}{ccc}
\alpha (1 - \tau) & \alpha (1 - \tau) & \dfrac{1}{2} + \rho + \dfrac{1}{2}\rho\tau \\
&&\\
\alpha \sqrt{\tau - 3} & - \alpha \sqrt{\tau - 3} & 0 \\
&&\\
\alpha & \alpha & -\dfrac{1}{2} - \dfrac{1}{2}\rho + \dfrac{1}{2}\tau + \rho\tau \\
\end{array}
\right)
\]
where $\rho = e^{\frac{2\pi}{3} i}= \dfrac{-1+\sqrt{-3}}{2}$, $\tau = \dfrac{1 + \sqrt{5}}{2}$ and $\alpha = \dfrac{1}{6} + \dfrac{1}{3}\rho + \dfrac{1}{3}\tau + \dfrac{1}{6}\rho\tau$. Then we can see that the homogeneous polynomial $F^{R}$ is a constant multiple of $W$. We call the coordinates $(x' : y' : z') = [R]^{-1} \cdot (x : y : z)$ the \textit{Wiman coordinates}.

We denote the conjugate $[R]^{-1} \mathcal{V} [R]$ by $\mathcal{V}'$. We define
\[
F_{\mathcal{V}'}(x, y, z) := W(x, y, z).
\]
In the Wiman coordinates, the homogeneous polynomials $\Phi_{\mathcal{V}}$ and $\Psi_{\mathcal{V}}$ are replaced by $\Phi_{\mathcal{V}'}$ and $\Psi_{\mathcal{V}'}$ where we define
\[
\Phi_{\mathcal{V}'}(x, y, z) := \det H(F_{\mathcal{V}'})(x, y, z)
\]
and 
\[
\Psi_{\mathcal{V}'}(x, y, z) := \det BH(F_{\mathcal{V}'}, \Phi_{\mathcal{V}'}) (x, y, z).
\]

Since the homogeneous polynomials $F_{\mathcal{V}'}$, $\Phi_{\mathcal{V}'}$ and $\Psi_{\mathcal{V}'}$ only have integer coefficients, the calculations for the claims is more easily and quickly done in the Wiman coordinates than in the former coordinates.
\end{remark}

Put \texttt{F} $= F_{\mathcal{V}'}(x, y, z)$ and let \texttt{JF} be the ideal generated by its derivatives. Next, set \texttt{Phi} $= \Phi_{\mathcal{V}'}(x, y, z)$(i.e., the determinant of the Hessian of \texttt{F}). Then we define \texttt{JPhi} to be the ideal generated by the derivatives of \texttt{Phi}. Let \texttt{BH} be the border Hesse matrix of \texttt{F} and \texttt{Phi}. Then we define \texttt{Psi} $= \Psi_{\mathcal{V}'}(x, y, z)$ (i.e. the determinant of \texttt{BH}) and \texttt{JPsi} the ideal generated by the derivatives of \texttt{Psi}.

\begin{lstlisting}
LIB "primdec.lib";
ring R = 0,(x,y,z),dp;

poly F = 10x3y3+9x5z+9y5z-45x2y2z2-135xyz4+27z6;
ideal JF = jacob(F);

poly Phi = det(jacob(jacob(F)));
ideal JPhi = jacob(Phi);

matrix HF = jacob(jacob(F));
matrix BH[4][4] = HF[1,1],HF[1,2],HF[1,3],JPhi[1],
                  HF[2,1],HF[2,2],HF[2,3],JPhi[2],
                  HF[3,1],HF[3,2],HF[3,3],JPhi[3],
                  JPhi[1],JPhi[2],JPhi[3],0;
poly Psi = det(BH);
ideal JPsi = jacob(Psi);
\end{lstlisting}

For the claim (4) and (5), we calculate the radicals of the ideal generated by \texttt{F}, \texttt{Phi} and \texttt{Psi} and the ideal generated by the derivatives of each of the homogeneous polynomials.
\begin{lstlisting}
//(4)
radical(ideal(F,Phi,Psi));

//(5)
radical(JF);
radical(JPhi);
radical(JPsi);
\end{lstlisting}
Each of these lines returns the following answer:
\begin{lstlisting}
_[1]=z
_[2]=y
_[3]=x
\end{lstlisting}
Thus each ideal is generated by $x$, $y$ and $z$.

Finally, we consider the claim (6). We check the (standard basis of) ideal generated by \texttt{F} and \texttt{Phi} and its radical.
\begin{lstlisting}
//(6)
std(ideal(F,Phi));
std(radical(ideal(F,Phi)));
\end{lstlisting}
Each of these lines returns the same $5$ generators.

\subsection{The icosahedral group}

In this subsection, we give the code for $G = \mathcal{I}$. We again use the Wiman coordinates in the calculation.

\begin{remark}
We define $\mathcal{I}' := [R]^{-1} \mathcal{I} [R]$. In the Wiman coordinates, $F_{\mathcal{I}}$ is transformed to a constant multiple of
\[
F_{\mathcal{I}'}(x, y, z) := xy + \eta z^{2}
\]
where $\eta = \dfrac{-3 + 3 \sqrt{-15}}{8}$. Since $\Phi_{\mathcal{I}} = F_{\mathcal{V}}$, we define $\Phi_{\mathcal{I}'} := W(x, y, z)$. In addition, the homogeneous polynomial $\Psi_{\mathcal{I}}$ is replaced by
\[
\Psi_{\mathcal{I}'} := \det BH(F_{\mathcal{I}'}, \Phi_{\mathcal{I}'})(x, y, z).
\]

As in the case of the Valentiner group, the calculation is more easily and quickly finished in the Wiman coordinates than in the former coordinates.
\end{remark}

We consider the ring generated by the variables $x$, $y$, $z$ and $e$. Set \texttt{g} the polynomial defined by $4e^{2} + 3 e + 9$. Note that one of the roots of \texttt{g} is $\eta$. We replace the constant $\eta$ by the variable $e$, and take the reduction modulo \texttt{g} from time to time.

The code is almost the same as the one for $G = \mathcal{V}$. We first define \texttt{F}, \texttt{Phi} and \texttt{Psi} and their Jacobian ideals.
\begin{lstlisting}
LIB "primdec.lib";
ring R = 0,(x,y,z,e),dp;
poly g = 4e2 + 3e + 9;

poly F = xy + ez2;
ideal JF = jacob(F);

poly Phi = 10x3y3+9x5z+9y5z-45x2y2z2-135xyz4+27z6;
ideal JPhi = jacob(Phi);

matrix HF = jacob(jacob(F));
matrix BH[4][4] = HF[1,1],HF[1,2],HF[1,3],JPhi[1],
                  HF[2,1],HF[2,2],HF[2,3],JPhi[2],
                  HF[3,1],HF[3,2],HF[3,3],JPhi[3],
                  JPhi[1],JPhi[2],JPhi[3],0;
poly Psi = det(BH);
ideal JPsi = jacob(Psi);
\end{lstlisting}

We check the claim (1). We define \texttt{X} to be the Jacobian of (\texttt{F}, \texttt{Phi}, \texttt{Psi}). Let \texttt{f} be the homomorphism $\mathbb{C}[x, y, z, e] \rightarrow \mathbb{C}[x, y, z, e]$ which sends a polynomial $P(x, y, z, n)$ to $P(1, 0, 0, e)$. Sends \texttt{X} with \texttt{f} and reduce the resulting polynomial by \texttt{g}.
\begin{lstlisting}
matrix J[3][3] = JF[1],JF[2],JF[3],
                 JPhi[1],JPhi[2],JPhi[3],
                 JPsi[1],JPsi[2],JPsi[3];
poly X = det(J);
map f = R,1,0,0,e;
reduce (f(X),g);
\end{lstlisting}
This return value is the nonzero constant $7290$. Thus we obtain that the Jacobian is nonzero at $(x, y, z) = (1, 0, 0)$.

For the claims (4) to (6), we may use almost the same code as in Subsection A.1 except that we have to  add \texttt{g} to the ideals.

For the claim (4) and (5), we use the following code.
\begin{lstlisting}
//(4)
radical(ideal(F,Phi,Psi) + ideal(g));

//(5)
radical(ideal(JF[1],JF[2],JF[3]) + ideal(g));
radical(ideal(JPhi[1],JPhi[2],JPhi[3]) + ideal(g));
radical(ideal(JPsi[1],JPsi[2],JPsi[3]) + ideal(g));
\end{lstlisting}
Each of these lines returns the following answer:
\begin{lstlisting}
_[1]=z
_[2]=y
_[3]=x
_[4]=4e2+3e+9
\end{lstlisting}
Since the forth generator is equal to \texttt{g}, each ideal is generated by $x$, $y$ and $z$.

We consider the claim (6). We use the following code.
\begin{lstlisting}
//(6)
std(ideal(F,Phi) + ideal(g));
std(radical(ideal(F,Phi) + ideal(g)));
\end{lstlisting}
Each of these lines returns the $9$ generators including \texttt{g}.

\subsection{The Klein group}

In this subsection, we check the claims for $G = \mathcal{K}$. The code is the same except for the definition of \texttt{F},  \texttt{Phi} and \texttt{Psi}. Note that \texttt{Phi} is $-54 \Phi_{\mathcal{K}}$ and \texttt{Psi} is $-26244 \Psi_{\mathcal{K}}$.
\begin{lstlisting}
LIB "primdec.lib";
ring R = 0,(x,y,z),dp;

poly F = x3y + y3z + z3x;
ideal JF = jacob(F);

poly Phi = det(jacob(jacob(F)));
ideal JPhi = jacob(Phi);

matrix HF = jacob(jacob(F));
matrix BH[4][4] = HF[1,1],HF[1,2],HF[1,3],JPhi[1],
                  HF[2,1],HF[2,2],HF[2,3],JPhi[2],
                  HF[3,1],HF[3,2],HF[3,3],JPhi[3],
                  JPhi[1],JPhi[2],JPhi[3],0;
poly Psi = det(BH);
ideal JPsi = jacob(Psi);
\end{lstlisting}

We check the claims (4) to (6). This is done with the same code as in Subsection A.1.
\begin{lstlisting}
//(4)
radical(ideal(F,Phi,Psi));

//(5)
radical(JF);
radical(JPhi);
radical(JPsi);

//(6)
std(ideal(F,Phi));
std(radical(ideal(F,Phi)));
\end{lstlisting}
For the claims (4) and (5), we also get the same answers for each of these lines and can see that each ideal is generated by $x$, $y$ and $z$. For the claim (6), for the ideal generated by \texttt{F} and \texttt{Phi} and its radical, we obtain the same $4$ generators.

\section*{Acknowledgements}

I would like to thank Associate Professor Nobuyoshi Takahashi for detailed advice in this paper.

\end{document}